\documentclass[final]{siamltex}

\usepackage{graphicx, amssymb, amsmath}
\usepackage{color}
\usepackage{pst-text}
\usepackage{soul} 
\usepackage{hyperref}

\newtheorem{example}{Example}[section]
\newtheorem{ex}{Exercice}
\newtheorem{rem}{Remark}[section]

\newtheorem{prop}{Proposition}[section]

\newcommand{\Z}{\mathbb{Z}}
\newcommand{\R}{\mathbb{R}}
\newcommand{\cA}{\mathcal{A}}

\newcommand{\cE}{\mathcal{E}}

\usepackage{epsf,epsfig}
\usepackage{graphicx, colordvi,amsmath,amsfonts,amssymb}
\usepackage{color}
\usepackage{tikz}
\usepackage{array}

\newcommand{\be}{\begin{eqnarray}}
\newcommand{\ee}{\end{eqnarray}}
\newcommand{\beno}{\begin{eqnarray*}}
\newcommand{\eeno}{\end{eqnarray*}}
\newcommand{\barr}[1]{\begin{array}{#1}}
\newcommand{\earr}{\end{array}}

\newcommand{\eps}{\epsilon}

\newcommand{\dt}{\tau}

\newcommand{\dx}{{\Delta x}}
\newcommand{\dy}{{\Delta y}}
\newcommand{\ma}{\alpha}

\newcommand{\disp}{\displaystyle}

\newcommand{\converge}{\rightarrow}

\newcommand{\ud}{\frac{1}{2}}

\if{
\newcount\Comments  
\Comments=1   
\usepackage{color}
\definecolor{darkgreen}{rgb}{0,0.5,0}
\definecolor{purple}{rgb}{1,0,1}
\newcommand{\kibitz}[2]{\ifnum\Comments=1\textcolor{#1}{#2}\fi}

}\fi

\newcommand{\figs}{.}
\newcommand{\figsnew}{.}

\title
{An efficient filtered scheme for some first order Hamilton-Jacobi-Bellman equations%
\thanks{The authors wish to acknowledge the support obtained by the following grants:
AFOSR Grant no.\ FA9550-10-1-0029, 
and by the EU under the 7th Framework Programme Marie Curie Initial Training Network
``FP7-PEOPLE-2010-ITN'', SADCO project, GA number 264735-SADCO}
}

\date{\today}

\author
{
Olivier Bokanowski,%
\footnote{Laboratoire Jacques-Louis Lions, 
Universit{\'e} Paris-Diderot (Paris 7)
UFR de Math{\'e}matiques - B{\^a}t. Sophie Germain
5 rue Thomas Mann
75205 Paris Cedex 13, e-mail:
\texttt{boka@math.jussieu.fr}}
\
Maurizio Falcone,%
\footnote{Dip. di Matematica La Sapienza Universit{\'a} di Roma 
 P. Aldo Moro, 5   00185 Roma - Italy, e-mail:   
\texttt{falcone@mat.uniroma1.it}}
\
Smita Sahu%
\footnote{Dip. di Matematica La Sapienza Universit{\'a} di Roma 
 P. Aldo Moro, 5   00185 Roma - Italy, e-mail: 
\texttt{sahu@mat.uniroma1.it}} 
}
\begin{document}
\maketitle

\begin{abstract}
We introduce a new class of ``filtered" schemes for some first order non-linear Hamilton-Jacobi-Bellman equations. 
The work follows recent ideas of Froese and Oberman (SIAM J. Numer. Anal., Vol 51, pp.423-444, 2013).
The proposed schemes are not monotone but still satisfy some $\epsilon$-monotone property. Convergence results and 
precise error estimates are given, of the order of $\sqrt{\dx}$ where $\dx$ is the mesh size.
The framework allows to construct finite difference discretizations that are easy to implement, high--order in the domains where the solution is smooth,
and provably convergent, together with error estimates.
Numerical tests on several examples are given to validate the approach, also showing how the filtered technique can be applied to stabilize 
an otherwise unstable high--order scheme.
\end{abstract}


\section{Introduction}
In this work, our aim is to develop high--order and convergent schemes for first order Hamilton-Jacobi (HJ) equations
of the following form
\be\label{eq:hj1}
  & & \partial_t v+H(x,\nabla v)=0, \quad  (t,x)\in[0,T]\times \R^d\\
  & & v(0,x)=v_0(x),\quad x\in\R^d.
\ee

\medskip\noindent
Basic assumptions on the Hamiltonian $H$ and the initial data $v_0$ will be introduced in the next section. 
It is well known that, in the one dimensional case, there is  a strong link between Hamilton-Jacobi equations and scalar conservation laws. Namely, the viscosity solution of the evolutive HJ equation is the primitive of the  entropy solution of the corresponding hyperbolic conservation law with the same hamiltonian. 
There are several schemes developed for hyperbolic conservation law
(see references~\cite{H83}~\cite{H84},~\cite{CL84},~\cite{GS98}). 
Most of the numerical ideas to solve hyperbolic conservation law can be extended to HJ equations.  
Well known high--order essentially non-oscillatory (ENO) scheme have been introduced 
by A. Harten et al.\ in~\cite{HEOC87} 
for hyperbolic conservation laws, and then extended to HJ equation by Osher and Shu~\cite{OS91}. 
ENO schemes have shown to have high--order accuracy although a precise convergence result is still missing and, for this property,
they  have been quite successful in many applications. 
Despite the fact that  there is no convergence proof of ENO schemes  towards the viscosity solution of \eqref{eq:hj1} in the general case,
convergence results may hold for related schemes, e.g. MUSCL schemes, as it has been proved by Lions and Souganidis in ~\cite{lio-sou-95}.
Convergence results of some non monotone scheme have also been shown in particular cases~\cite{bok-meg-zid-2010}.
Another interesting result has been proved by 
Fjordholm et al.~\cite{FMT12}, they have shown that ENO interpolation is stable but the stability result 
is not sufficient to conclude total variation boundedness (TVB) of the ENO reconstruction procedure. 
In~\cite{F13}, a conjecture related to weak total variation property for ENO schemes is given.
Finally, let us also mention~\cite{CFR05}, where weighted essentially non-oscillatory (WENO) schemes have been applied to HJ equations;
the convergence proof of the scheme relies also on the work of Ferretti~\cite{fer-02} where higher than first order
schemes are proposed in a semi-Lagriangian setting, yet with restrictive conditions on the mesh steps.

In this paper we give a very simple way to construct high--order schemes in a convergent framework.
It is known (by Godunov's theorem) that a monotone scheme can be at most first order. 
Therefore it is needed to look for non-monotone schemes. The difficulty is then to combine non-monotonicity of the scheme
and convergence towards the viscosity solution of \eqref{eq:hj1}, and also to obtain error estimates. 
In our approach we will adapt a general idea of Froese and Oberman~\cite{FO13},
that was presented for stationnary second order Hamilton-Jacobi equations and based on the use of a ``filter" function.
Here we focus 
mainly on the case of evolutive first order Hamilton-Jacobi equation \eqref{eq:hj1}, and 
an adaptation to the steady case will be also considered.
We will design a slightly different filter function for which the filtered scheme is still 
an ``$\epsilon$-monotone" scheme (see Eq.\ref{eq:eps-monotone}), but that improves the numerical results.
Let us mention also the work \cite{bok-fal-fer-gru-kal-zid} for steady equations where some
$\eps$-monotone semi-Lagrangian schemes are studied.

The paper is organized as follows. 
In Section 2, we present the schemes and give main convergence results.
Section 3 is devoted to the numerical tests on several academic examples to illustrate our approach in one and 
two-dimensional cases. A test on nonlinear steady equations , as well an evolutive ``obstacle" HJ equation in the form of
$\min(u_t + H(x,u_x), u-g(x))=0$ for a given function $g$ are also included in this section.
Finally, Section 4 contains our concluding remarks.

\section{Definitions and main results}

\subsection{Setting of the problem}
Let us denote by  $|\cdot|$ the Euclidean norm on $\R^d$ ($d\geq 1$).
The following  classical assumptions will be considered in the sequel of this paper: \\
{\bf (A1)} 
$v_{0}$ is Lipschitz continuous function i.e. there exist $L_0 >0$ such that for every $x~,y\in \R^d$,
\begin{equation}\label{eq:L0}
  |v_0(x)-v_0(y)|\leq L_0 |x-y|.
\end{equation}
\medskip\noindent
{\bf (A2)} $H:\R^d\times\R^d\rightarrow \R^d$ 
satisfies, for some constant $C\geq 0$, for all $p,q,x,y\in \R^d$:
\be
    |H(y,p) - H(x,p)|\leq C (1+|p|) |y-x|,
\ee
and 
\be
    |H(x,q) - H(x,p)|\leq C (1+|x|) |q-p|.
\ee

Under assumptions (A1) and (A2) there exists a unique viscosity solution for \eqref{eq:hj1}
(see Ishii \cite{Ish-84}).
Furthermore $v$ is locally Lipschitz continuous on $[0,T]\times \R^d$.

For clarity of presentation we focus on the one-dimensional case and consider the
following simplified problem:
\be\label{eq:hj}
&&v_t + H(x, v_x)=0, \quad x\in \R,\ t\in [0,T],\\
&&v(0,x)=v_0(x),\quad x\in \R.
\ee

\subsection{Construction of the filtered scheme}
Let $\dt = \Delta t >0$ be a time step (in the form of $\dt=\frac{T}{N}$ for some $N\geq 1$), and $\dx>0$ be a space step. A uniform mesh in time  is defined by
$t_n:=n\dt$, $n\in[0,\dots,N]$, and in space by the nodes $x_j:=j\dx$, $j\in \Z$.

The construction of a filtered scheme needs three ingredients: 
\begin{itemize}
\item
  a monotone scheme, denoted $S^M$
\item 
  a high--order scheme, denoted $S^A$
\item 
  a bounded ``filter" function, $F: \R\converge \R$.
\end{itemize}
The high-order scheme need not be convergent nor stable;
the letter $A$ stands for ``arbitrary order", following \cite{FO13}.
For a start, $S^M$ will be based on a finite difference scheme. Later on we will also propose 
a definition of $S^M$ based on a semi-Lagriangian scheme.

Then, the filtered scheme $S^F$ is defined as
\be\label{eq:FS}
  u^{n+1}_j \equiv S^F(u^n)_j := S^{M}(u^n)_j+\epsilon\dt F\bigg(\frac{S^{A}(u^n)_j-S^{M}(u^n)_j}{\epsilon\dt}\bigg),
\ee
where $\eps=\eps_{\dt,\dx}>0$ is a parameter satisfying
\be\label{eq:epsgotozero}
  \lim_{(\dt,\dx)\converge 0} \eps = 0.
\ee
More hints  on the choice of $\eps$ will be given later on.

The scheme is initialized in the standard way, i.e. 
\be \label{eq:FDexplicit_b} 
   u_j^0 := v_0(x_j),\quad \forall j\in \Z.
\ee

Now we make precise some requirements on $S^M$, $S^A$ and the function $F$.

\medskip\noindent
{\em Definition of the monotone finite difference scheme $S^M$} \\
Following Crandall and Lions~\cite{CL84}, we consider a finite difference scheme written as $u^{n+1}=S^M(u^n)$ with
\be\label{eq:FD}
   S^{M}(u^{n})(x) := u^{n}(x) -\dt\ h^M(x,D^-u^n(x),D^+u^n(x)),
\ee
with 
$$ 
   D^\pm u(x) :=\pm \frac{u(x\pm \dx) - u(x)}{\dx},
$$
where $h^M$ corresponds to a monotone numerical Hamiltonian that will be made precise below.
We will denote also $S^M(u^n)_j := S^M(u^n)(x_j)$.
Therefore the scheme also reads,
for all $j\in \Z$, $\forall n\geq 0$:
\be \label{eq:FDexplicit_a} 
  u^{n+1}_j:= u^{n}_j -\dt\ h^M(x_j,D^-u^n_j,D^+u^n_j), \quad D^{\pm} u^n_j:=\pm \frac{u^n_{j\pm 1}-u^n_j}{\dx}.
\ee

\medskip\noindent
{\bf (A3)  Assumptions on $S^M$}\\[0.0cm]
We will use the following assumptions throughout this paper:\\
\begin{tabular}{ll} 
& $(i)$ $h^M$ is a Lipschitz continuous function.\\
& $(ii)$  (\textit{consistency}) 
  $\forall x$, $\forall u$,\ $ h^M(x,u,u)=H(x,u).$
  \\
& $(iii)$ (\textit{monotonicity}) for any functions $u,v$,\\
&  \hspace{2cm}
  $
   \mbox{$u\leq v$ $\Longrightarrow$  $S^M(u)\leq S^M(v)$}.
  $
\end{tabular}


In pratice condition (A3)-$(iii)$ is only required at mesh points and the condition reads
\be \label{eq:monotcond}
  u_j\leq v_j,\ \forall j, \quad \Rightarrow\quad  S^M(u)_j \leq S^M(v)_j,\ \forall j.
\ee

At this stage, we notice that under condition (A3) the filtered scheme is ``$\eps$-monotone" in the sense that
\be\label{eq:eps-monotone}
 \ u_j\leq v_j,\ \forall j, \quad \Rightarrow\quad  S^M(u)_j \leq S^M(v)_j + \eps\tau,\ \forall j.
\ee
with $\eps\converge 0$ as $(\dt,\dx)\converge 0$. This implies the convergence of the scheme
by  Barles-Souganidis convergence theorem (see \cite{BS91,A09}).

\begin{rem}
Under assumption $(i)$, the consistency property $(ii)$ is equivalent to say that,
for any $v\in C^2([0,T]\times\mathbb{R})$, there exists a constant $C_M\geq 0$ independant of $\dx$ such that
\be
  \label{eq:SM-consist}
  \bigg|  h^M(x,D^-v(x),D^+ v(x)) - H(x,v_x) \bigg| \leq C_M \dx \| \partial_{xx} v  \|_\infty.
\ee
The same statement holds true if \eqref{eq:SM-consist} is replaced by the following consistency error estimate:
\be
   & & \hspace*{-2cm} 
      \cE_{S^M}(v)(t,x) :=
      \bigg|  \frac{v(t+\dt,x)-S^M(v(t,.))(x)}{\dt} - \big( v_t(t,x) + H(x,v_x(t,x))) \bigg|
      \nonumber \\
   & & \qquad \leq \ C_M \bigg(\dt \| \partial_{tt} v  \|_\infty + \dx \| \partial_{xx} v \|_\infty\bigg).
\ee
\end{rem}

\begin{rem}
Assuming $(i)$, it is easily shown that the monotonicity property $(iii)$ is equivalent to say that $h^M=h^M(x,u^-,u^+)$ 
satisfies, a.e. $(x,u^-,u^+ )\in \R^3$:
\be\label{eq:MONOT-1}
  \mbox{$\disp\frac{\partial h^M}{\partial u^-} \geq 0$,\ \ \ $\disp\frac{\partial h^M}{\partial u^+} \leq 0$,}
\ee
and the CFL condition
\be\label{eq:MONOT-2}
 \frac{\dt}{\dx} \bigg(\frac{\partial h^M}{\partial u^-}(x,u^-,u^+) - \frac{\partial h^M}{\partial u^+}(x,u^-,u^+)\bigg) \leq 1.
\ee
When using finite difference schemes, 
it is assumed that the CFL condition \eqref{eq:MONOT-2} 
is satisfied, and that can be written equivalently 
in the form 
\be\label{eq:CFL}
  c_{0} \frac{\dt}{\dx} \leq 1,
\ee
where $c_0$ is a constant independant of $\dt$ and $\dx$.
\end{rem}

\begin{example} 
Let us consider the Lax-Friedrichs numerical Hamiltonian is 
$$
  h^{M,LF}(x,u^-,u^+):=H(x,\frac{u^-+u^+}{2}) - \frac{c_0}{2}(u^+-u^-)
$$
where $c_0>0$ is a constant.
The scheme is consistant; it is furthermore monotone under the conditions
$\max_{x,p}|\partial_p H(x,p)|\leq c_0$, and $c_0\frac{\dt}{\dx}\leq 1$.
\end{example}


%
%
%

\medskip\noindent
{\em Definition of the high--order scheme $S^A$:}
we consider an iterative scheme of ``high--order" in the form $u^{n+1}=S^A(u^n)$,
written as
\begin{small}
\be \label{eq:SA-developped}
  S^{A}(u^{n})(x)= u^{n}(x)-\tau  h^{A}(x,D^{k,-}u^n(x),\dots, D^-u^n(x),D^+u^n(x),\dots,D^{k,+}u^n(x)),\nonumber
\ee
\end{small}
\!\!where $h^A$ corresponds to a ``high-order" numerical Hamiltonian,
and $D^{\ell,\pm}u(x):=\pm \frac{u(x\pm \ell\dx)-u(x)}{\dx}$ for $\ell=1,\dots,k$. To simplify the notations 
we may write \eqref{eq:SA-developped} in the more compact form
\be\label{eq:SA}
  S^{A}(u^{n})(x)= u^{n}(x)-\tau  h^{A}\big(x, D^{\pm} u^n(x))
\ee
even if there is a dependency on $\ell$  in  $(D^{\ell,\pm} u^n(x))_{\ell=1,\dots,k}$.



\medskip\noindent
{\bf (A4)  Assumptions on $S^A$:}\\[0.2cm]
We will use the following assumptions:\\
{\em 
\indent $(i)$ $h^A$ is a Lipschitz continuous function.\\[0.2cm]
\indent $(ii)$  (\textit{high--order consistency}) There exists $k\geq 2$, for all $\ell\in[1,\dots,k]$,
for any function $v=v(t,x)$ of class $C^{\ell+1}$, there exists $C_{A,\ell}\geq 0$,
\be\label{eq:H1consist_a}
  & & \hspace*{-2cm}  
    \cE_{S^A}(v)(t,x) :=
    \bigg|  \frac{v(t+\dt,x)-S^A(v(t,.))(x)}{\dt} -\big( v_t(t,x) + H(x,v_x(t,x))) \bigg|
    \\
  & & \leq \ C_{A,\ell} \bigg (\dt^\ell \|\partial_t^{\ell+1} v\|_\infty +  
                        \dx^\ell \|\partial_x^{\ell+1} v\|_\infty \bigg).
\ee
}
Here $v^\ell_{x}$ denotes the $\ell$-th derivative of $v$  w.r.t.\ $x$. 

\begin{rem}\label{rem:hA-implications}
The high-order consistency implies, 
for all $\ell \in [1,\dots,k]$, and for $v \in C^{\ell+1}(\R)$, 
\beno\label{eq:H1consist-implications}
  \bigg| h^A(x,\dots,D^-v,D^+v,\dots) - H(x,v_x)\bigg|  
  \leq \ C_{A,\ell}  \|\partial_x^{\ell+1} v\|_\infty  \dx^\ell.
\eeno
\end{rem}

\begin{example}\label{rem:CFD-RK2}
{\em (Centered scheme)}
A typical example with $k=2$ is obtained with the centered TVD (Total Variation Diminishing) approximation in space and
the Runge-Kutta 2nd order scheme in time (or Heun scheme):
\begin{subequations} \label{eq:SA-RK2}
\be\label{eq:SA-RK2a}
  & & S_0(u^n)_j:= u^n_j - \dt H(x_j,\frac{u^n_{j+1}-u^n_{j-1}}{2\dx}),
\ee
and
\be\label{eq:SA-RK2b}
  & & S^A(u):= \frac{1}{2}(u + S_0 (S_0(u))).
\ee
\end{subequations}
Of course there is no reason for the centered scheme to be stable (as it will be shown in the numerical section).
Using a filter will help stabilize the scheme.
\end{example}
A similar example with $k=3$ can be obtained with any third order finite difference approximation in space
and the TVD-RK3 scheme in time~\cite{got-shu-98}.

\medskip\noindent
{\em Definition of the filter function $F$.}
We recall that Froese and Oberman's filter function used in~\cite{FO13} is:
\beno
  \tilde F(x) = sign(x)\max(1-||x|-1|,0) = \left\{
  \begin{array}{l l}
    x & \quad |x| \leq 1.\\
    0 & \quad |x|\geq 2.\\
    -x+2 &\quad 1 \leq x\leq 2.\\
    -x-2 &\quad  -2\leq x \leq -1.\\
  \end{array} \right.
\eeno
In the present work we define a new filter function simply as follows:
\be \label{def:Filter}
  F(x) := x 1_{|x|\leq 1} =  \left\{
  \begin{array}{l l}
    x & \quad \text{if} \ |x| \leq 1,\\
    0 & \quad \text{otherwise}.\\
      \end{array} \right.
\ee

\begin{figure}[!h]
\includegraphics[width=0.5\textwidth]{\figs/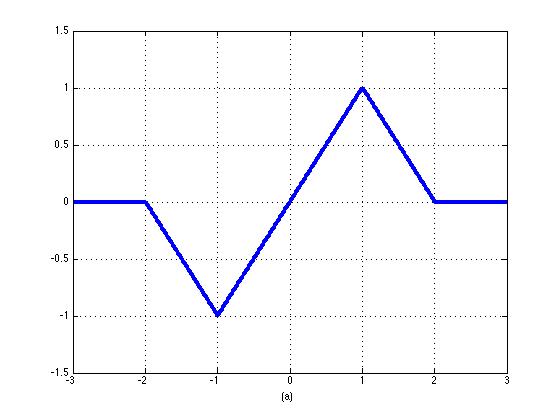}
\includegraphics[width=0.5\textwidth]{\figs/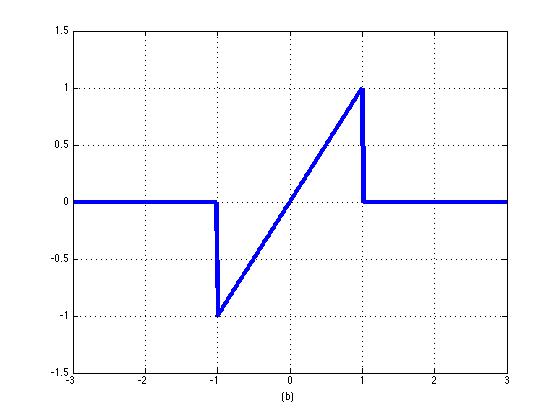}\\
\caption{Froese and Oberman's filter (left), new filter (right)}
\end{figure}

The idea of the present filter function is to keep the high--order scheme when $|h^A-h^M|\leq \eps$
(because then $|S^A-S^M|/(\tau \eps) \leq 1$ and $S^F=S^M + \tau\eps F(\frac{S^A-S^M}{\tau\eps}) \equiv S^A$),
whereas $F=0$ and $S^F=S^M$ if that bound is not satisfied, i.e., the scheme is simply given by the monotone scheme itself.
Clearly the main difference is the discontinuity at $x=-1, 1$.

%

\subsection{Convergence result}
The following theorem gives several basic convergence results for the filtered scheme.
Note that the high-order assumption (A4) will not be necessary to get 
the error estimates $(i)$-$(ii)$. It will be only used to get a high-order consistency error estimate in the regular case
(part $(iii)$). Globally the scheme will have just an $O(\sqrt{\Delta x})$ rate of convergence for just Lipschitz continuous 
solutions because the jumps in the gradient prevent high-order accuracy on the kinks.

\begin{theorem} \label{th:estimate1}
Assume (A1)-(A2), and let $v_0$ be bounded. 
We assume also that $S^M$ satisfies (A3), and $|F|\leq 1$.
Let $u^n$ denote the filtered scheme~\eqref{eq:FS}.
Let $v^n_j:=v(t_n,x_j)$ where $v$ is the exact solution of~\eqref{eq:hj}. 
Assume
\be \label{eq:eps-1} 
  0<\eps \leq c_0\sqrt{\dx}
\ee
for some constant $c_0>0$.

$(i)$ The scheme $u^n$ satisfies the Crandall-Lions estimate
\be
  \|u^{n} - v^{n}\|_{\infty} \leq C \sqrt{\dx},  \quad \forall\ n=0,...,N. 
\ee
for some constant $C$ independent of $\dx$.

$(ii)$ (First order convergence for classical solutions.)
If furthermore the exact solution $v$ belongs to $C^2([0,T]\times\R)$,
and $\eps\leq c_0 \dx$ (instead of \eqref{eq:eps-1}),
then,  we have 
\be
  \|u^{n} - v^{n}\|_{\infty} \leq C \dx,  \quad  n=0,...,N,
\ee
for some constant $C$ independent of $\dx$.

$(iii)$ (Local high-order consistency.)
Assume that $S^A$ is a high--order scheme satisfying (A4) for some $k\geq 2$. 
Let $1\leq \ell \leq k$ and $v$ be a $C^{\ell+1}$ function 
in a neighborhood of a point $(t,x)\in (0,T)\times \R$.
Assume that
\be
  \mbox{$(C_{A,1} + C_M)\ \bigg(\|v_{tt}\|_\infty\, \tau + \| v_{xx}\|_\infty\, \dx \bigg) \leq \eps$.}
\ee
Then, for sufficiently small $t_n-t$, $x_j-x$, $\dt$, $\dx$,
it holds 
$$
  {S^F}(v^n)_j = {S^A}(v^n)_j
$$
and, in particular, a local high-order consistency error for the filtered scheme $S^F$ holds:
$$ {\cE}_{S^F}(v^n)_j \equiv \cE_{S^A}(v^n)_j =  O(\dx^\ell) $$
(the consistency error $\cE_{S^A}$ is defined in \eqref{eq:H1consist_a}).
\end{theorem}

\begin{proof}
$(i)$ Let $w^{n+1}_j=S^M(w^n)_j$ be defined with the monotone scheme only, with $w^0_j=v_0(x_j)=u^0_j$.
By definitions,
\beno
  u^{n+1}_j - w^{n+1}_j
  & = & S^{M}(u^{n})_j - S^M(w^n)_j +\eps \dt F\big(.) 
\eeno
Hence, by using the monotonicity of $S^M$, 
\beno
  \max_j |u^{n+1}_j - w^{n+1}_j| 
  & \leq  & \max_j |u^{n}_j - w^n_j| +\eps \dt,
\eeno
and by recursion, for $n\leq N$,
\beno
  \max_j |u^{n}_j - w^{n}_j|  
  & \leq  & \eps n \dt  \leq  T \eps.
\eeno
On the other hand,  by Crandall and Lions \cite{CL84}, an error estimate holds for the monotone scheme:
\beno 
  \max_j |w^{n}_j-v^n_j| & \leq  & C \sqrt{\dx},
\eeno   
for some $C\geq 0$.
By summing up the previous bounds, we deduce
\beno
  \max_j |u^{n}_j - v^{n}_j|  
  & \leq  & C\sqrt{\dx} +  T \eps,
\eeno
and together with the assumption on $\eps$, it gives the desired result.

$(ii)$  
Let 
$\displaystyle \cE^n_j:=\frac{v^{n+1}_j - S^M(v^n)_j}{\dt}$.
If the solution is $C^2$ regular with bounded second order derivatives, 
then the consistency error is bounded by 
\be\label{eq:Epsnj}
  | \cE^n_j | \leq C_M(\dt + \dx).
\ee
Hence
\beno
 | u^{n+1}_j - v^{n+1}_j|
  & = & | S^{M}(u^{n})_j - S^M(v^n)_j + \dt \cE^n_j + \dt \eps F\big(.)  |\\
  & \leq & 
  \| u^{n} - v^n\|_\infty + \dt \|\cE^n\|_\infty + \dt \eps.
\eeno
By recursion, for $n\dt \leq T$,
\beno
 \|u^{n} - v^{n}\|_\infty
  & \leq & \| u^{0} - v^0\|_\infty + T ( \max_{0\leq k\leq N-1}\|\cE^k\|_\infty + \eps).
\eeno
Finally by using the assumption on $\eps$, the bound \eqref{eq:Epsnj} 
and the fact that $\dt=O(\dx)$ (using CFL condition \eqref{eq:CFL}),
we get the desired result.

$(iii)$
To prove that $S^F(v^n)_j=S^A(v^n)_j$, one has to check that 
$$
   \frac{|S^A(v^n)_j-S^M(v^n)_j|}{\eps\tau} \leq 1
$$
as $(\dt,\dx)\converge 0$.
By using the consistency error definitions, 
\if{
\beno
   \frac{|S^A(v^n)_j-S^M(v^n)_j|}{\tau}
  &  = & 
   \bigg| \frac{v^{n+1}_j - S^A(v^n)_j}{\tau} + v_t(t_n,x_j) + H(x_j,v_x(t_n,x_j)) \\
   &  & \quad     - \bigg( \frac{v^{n+1}_j - S^M(v^n)_j}{\tau}+ v_t(t_n,x_j) + H(x_j,v_x(t_n,x_j))\bigg) \bigg| \\
  &  \leq  &  |\cE_{S^A}(v^n)_j| + |\cE_{S^M}(v^n)_j| \\
  &  \leq  &  (C_{A,1} + C_M) (\tau \|v_{tt}\|_\infty + \dx \| v_{xx}\|_\infty) 
\eeno
}\fi
\beno
   \frac{|S^A(v^n)_j-S^M(v^n)_j|}{\tau}
  &  = & 
    \bigg| \frac{v^{n+1}_j - S^M(v^n)_j}{\tau}+ v_t(t_n,x_j) + H(x_j,v_x(t_n,x_j)) \\
   &  & \quad     - \bigg( \frac{v^{n+1}_j - S^A(v^n)_j}{\tau} + v_t(t_n,x_j) + H(x_j,v_x(t_n,x_j))\bigg) \bigg| \\
  &  \leq  &  |\cE_{S^M}(v^n)_j| + |\cE_{S^A}(v^n)_j|  \\
  &  \leq  &  (C_{A,1} + C_M) (\tau \|v_{tt}\|_\infty + \dx \| v_{xx}\|_\infty) 
\eeno
Hence 
the desired result follows.
\end{proof}



\begin{rem}\label{rem:PROJECTION} {\bf Other approaches.}
It is already known from the original work of Osher and Shu~\cite{OS91} that it is possible to modify an ENO scheme
in order to obtain a convergent scheme. 
For instance, if $D^{\pm,A} u^n_j$ denotes a high--order finite difference derivative estimate (of ENO type),
a projection on the first order finite difference derivative $D^\pm u^n_j$ can be used, 
up to a controlled error (see in particular Remark 2.2 of \cite{OS91}):
$$ \mbox{instead of \ $D^{\pm,A} u^n_j$,} \quad \mbox{use} \  P_{[D^\pm u^n_j, M\dx]}(D^{\pm,A} u^n_j) $$
where $P_{[a,b]}(y)$ is the projection defined by:
$$
  P_{[a,b]}(y):=\left\{ \barr{ll}
    y   & \mbox{if\ \ $a-b\leq y\leq a+b$}\\
    a-b & \mbox{if\ \ $y\leq a-b$}\\
    a+b & \mbox{if\ \ $y\geq a+b$}
    \earr \right.
$$
and $M>0$ is some constant greater than the expected value $\frac{1}{2} |u_{xx}(t_n,x_j)|$.
However, we emphasize that in our approach we do not consider a projection but a perturbation with a filter, which 
is sligthly different.
Indeed, by using a projection into an interval of the form $[a-M\dx,\, a+M\dx]$ where $a=D^\pm u^n_i$,
numerical tests show that we may choose too often one of the extremal values $a\pm M\dx$ which 
is then produces an overall too big error (worse than using the first order finite differences).

Following the present approach, we would rather advice to use, 
$$ \mbox{instead of \ $D^{\pm,A} u^n_j$,} \quad \mbox{the value} \  \tilde P_{[D^\pm u^n_j, M\dx]}(D^{\pm,A} u^n_j) $$
where $\tilde P_{[a,b]}(y)$ is defined by:
$$
  \tilde P_{[a,b]}(y):=\left\{ \barr{ll}
    y   & \mbox{if\ \ $a-b\leq y\leq a+b$}\\
    a & \mbox{if\ \ $y\notin [a-b,a+b]$}\\
    \earr \right.
$$
\end{rem}

\begin{rem}{\bf Filtered semi-Lagrangian scheme.}
Let us consider the case of

\be \label{eq:h}
  & &  H(x,p):= \min_{b\in B} \max_{a\in A}\{-f(x,a,b).p- \ell(x,a,b)\},
\ee
where $A\subset \R^m$ and $B\subset \R^n$ are non-empty compact sets (with $m,n \geq 1$), 
\if{
and $f: \R^d\times {A} \rightarrow \R^d$ 
and $\ell: \R^d\times {A} \rightarrow \R$ are Lipschitz continuous w.r.t.\ $x$:
$\exists L\geq 0$, $\forall a\in {A}$, $\forall x,y$:
\be \label{eq:Lff}
   \max(|f(x,a)- f(y,a)|,|\ell(x,a)-\ell(y,a)|)\ \leq \ L |x-y|.
\ee
}\fi
$f: \R^d\times {A}\times B \rightarrow \R^d$ 
and $\ell: \R^d\times {A}\times B \rightarrow \R$ are Lipschitz continuous w.r.t.\ $x$:
$\exists L\geq 0$, $\forall (a,b)\in {A}\times B$, $\forall x,y$:
\be \label{eq:Lff}
   \max(|f(x,a,b)- f(y,a,b)|,|\ell(x,a,b)-\ell(y,a,b)|)\ \leq \ L |x-y|.
\ee
(We notice that (A2) is satisfied for hamiltonian functions such as \eqref{eq:h}.)
\if{
$$ 
  |H(x_1,p)-H(x_2,p)|\leq L(1+|p|)\,|x_1-x_2|
$$
and 
$$ 
  |H(x,p)-H(x,q)|\leq |p-q|\, \max_{a,b} |f(x,a,b)| \leq |p-q|\, (\max_{a,b}|f(0,a,b)| + L|x|).
$$
}\fi
Let $[u]$ denote the $P^1$-interpolation of $u$ in dimension one on the mesh $(x_j)$, i.e. 
\be\label{interpolation}
   x\in[x_j,x_{j+1}] \quad \Rightarrow \quad 
  [u](x) := \frac{x_{j+1}-x}{\dx} u_j+\frac{x-x_{j}}{\dx} u_{j+1}.
\ee
Then a monotone SL scheme can be defined as follows:
\be\label{eq:sl}
   S^M(u^n)_j := \displaystyle \min_{a\in A} \max_{b\in B} 
    \bigg( [u^{n}]\big(x_j+ \dt f(x_j,a,b)\big)+ \dt \ell (x_j,a,b) \bigg).
\ee 
A filtered scheme based on SL can then be defined by \eqref{eq:FS} and \eqref{eq:sl}.
Convergence result as well as error estimates could also be obtained in this framework.
(For error estimates for the monotone SL scheme, we refer to~\cite{sor-98, FF14}.)
\end{rem}

\if{
\begin{rem}
In the case when $D(x)$ is convex then the scheme takes the more usual form
\be\label{eq:sl-b}
  S^M(u^n)_j = \displaystyle \min_{a\in A}[u^{n}](x_j+ \dt f(x_j,a)+ \dt \ell (x_j,a)).
\ee 
\end{rem}
}\fi

\if{
\begin{prop}
(i) Assume (A1) ($v_0$ is Lipschitz continuous). 
Assume furthermore that\\
\indent - if $f=f(x,a)$: $\forall x$, $f(x,A)$ is convex.\\
\indent - if $f=f(x,a,b)$: we assume $f(x,a,b)=g_1(x,b)\cdot A  + g_2(x,b)$ with $A$ convex, and $B$ compact.\\
Then the following consistency error holds, $\forall n\leq N$, $\forall j$:
$$ |\cE_{S^M}(v^n)_j| \leq C (\tau + \frac{\dx}{\tau}).
$$
\end{prop}
\begin{proof}
Let us assume that $f(x,a,b)=f(x,a)$ and $\ell(x,a,b)=\ell(x,a)$, the arguments in the general case beeing very similar.
Equation \eqref{eq:hj} is the Bellman equation for a finite horizon control problem 
(see ~\cite{BCD97}, Ch. III  for details). 
More precisely, 
we  consider the controlled system (whose solution is denoted $y=y_x^\alpha$)
\be\label{eq:csys}
  \dot y(t)=f(y(s),\alpha(s)), \quad \text{ for a.e.}\quad s\geq 0 \quad y(0)=x,
\ee
where $\alpha\in \cA$, the set of  measurable functions: $(0,\infty) \rightarrow A$
(Caratheodory solutions are considered in~\eqref{eq:csys}).
Let the value function $v$ be defined by
\be\label{v}
  v(t,x) :=\inf_{\ma\in\cA} \ \ v_{0}(y_{x}^{\alpha}(t))+\int_0^t \! \ell(y_{x}^{\alpha}(s),\ma(s))\, ds.
\ee
Then $v$ satisfies, for $h\geq 0$:
\be\label{eq:vdpp}
  v(t+h,x)=\inf_{\ma\in \cA} v(t,y^\ma_x(h)) + \int_{0}^{h} \ell(y^\ma_x(s),\ma(s))\, ds.
\ee
It can be shown that $v$ is the unique solution of \eqref{eq:hj} in the viscosity sense. We will use the formula 
\eqref{eq:vdpp} in order to get the consistency error.

From the definition of $v$ and using Gronwall's Lemma we first obtain that 
$|v(t,x_2)-v(t,x_1)|\leq L_{v_0} e^{LT} |x_2-x_1|$ where $L_{v_0}$ is the Lipschitz constant of $v_0$.
The Lipschitz property for $v(t_n,.)=v^n$ implies a control on the interpolation error:
\be
  \|[v^n]-v^n\|_\infty \leq L_{v_0} e^{LT} \dx.
\ee

We first consider constant controls $\ma(s)=a\in A$, for $s\in [0,h]$. We note that $y_x^a(h)=x + h f(x,a) + O(h^2)$,
and therefore it is then deduced from \eqref{eq:vdpp} and the Lipschitz property of $v(t_n,.)$ that
$$ v^{n+1}_j \leq \sup_{a\in A} v^n(x_j + h f(x_j,a)) + h  \ell(x_j,a) + Ch^2,$$
for some constant $C\geq 0$.
Hence 
\be v^{n+1}_j - S^M(v^n)_j
  & \leq & 
      \sup_{a\in A} \bigg( v^n (x_j + \dt f(x_j,a)) + \dt \ell(x_j,a) \bigg) \nonumber\\
  & & \hspace{1cm}  - \sup_{a\in A} \bigg([v^n](x_j + \dt f(x_j,a)) + \dt \ell(x_j,a) \bigg) + C\dt^2 \nonumber \\
  & \leq & \| v^n-[v^n] \|_\infty + C\dt^2\\
  & \leq &  L_{v_0} e^{LT} \dx + C\dt^2.
\ee
This gives the desired upper bound for error consistency.

Now we want to prove a reverse inequality. Let $C>0$.  
There exists a sequence of controls $\ma_h$ such that
$$
  v^{n+1}_j \geq v^n(y_x^{\ma_h}(h)) + \int_0^h \ell(y_x^{\ma_h}(s),\ma_h(s))ds - C h^2.
$$
\begin{center}
\fbox{to be continued ...}
\end{center}
\end{proof}
}\fi






\if{
\begin{theorem}[{\bf convergence of filtered SL schemes}]
Let assumptions (A1)-(A2'), and $v_0$ bounded.
Then  for all $n\dt \leq T$, 
\beno
  \max_j |v^n(x_j)-u^n(x_j)| \leq C (\eps + \dt+\frac{\dx}{\dt}).
\eeno
In particular by choosing $\eps\leq c_0\sqrt{\dx}$ and $\dt= c_1\sqrt{\dx}$ the error behaves as $O(\sqrt{\dx})$.
\end{theorem} 
\begin{proof}
As in the proof for the filtered finite difference schemes,
if $w^{n+1}_j=S^M(w^n)_j$ is defined with the monotone SL scheme, and with $w^0_j=v_0(x_j)$,
then, for $n\leq N$,
\beno
  \max_j |u^{n}_j - w^{n}_j|  
  & \leq  & \eps n \dt \  \leq \  T \eps.
\eeno
There remains to find an error estimate for $|v^n_j - w^n_j|$.

It can be proved that $u^n$ is Lipschitz regular. Then, the error between the semi-discrete scheme (say $\vartheta^n$)
and the fully discrete scheme is bounded by $O(\frac{\dx}{\dt})$. Finally, it is known, for the semi-discrete scheme 
that 

\begin{center}
\fbox{CHECK : Classical estimate for monotone SL scheme for games with finite horizon}
\fbox{(Maurizio)}
\end{center}
$$
   \max_j |v^n_j - \vartheta^n_j| \leq C \sqrt{\tau}. 
$$
{\color{blue}
  (NB for $H(x,p)=max_{a\in A} (-f(x,a) p - \ell(x,a))$ I see a direct bound in $O(\tau)$, for games I am not sure).
}

Hence the global bound 
$$
   \max_j |v^n_j - w^n_j| \leq C (\sqrt{\tau} + \frac{\dx}{\tau}).
$$ 
\end{proof}
}\fi

\if{
\beno\label{eq:sl1}
\|u^{n} - v^{n}\|_{\infty}
  \leq \|S^{M}(u^{n-1})-S^{M}(v^{n-1})\|_{\infty}+\|S^{M}(v^{n-1})-v^{n}\|_\infty+\eps\dt \|F\|_\infty.
\eeno
By using the monotonicity of $S^M$, we get 
\beno
\|u^{n} - v^{n}\|_{\infty}\leq \|u^{n-1}-v^{n-1}\|_{\infty}+\|S^{M}(v^{n-1})-v^{n}\|_\infty+\eps \dt \|F\|.\\
\eeno
By the definition of filter function we know that $\|F\|\leq 1$, and consistency error estimate \eqref{eq:H1consist_c} of semi-Lagrangian scheme \eqref{eq:sl}, we get 
\beno
  \|u^{n} - v^{n}\|_{\infty}&\leq \|u^{n-1}-v^{n-1}\|_{\infty} + L \big(\dt+\frac{\dx}{ \dt}\big )\dt+\eps \dt.
\eeno
Hence by recursion, we get
\beno
  \|u^{n} - v^{n}\| _{\infty}& \leq \|u^{0} - v^{0}\|_{\infty} + LT \big(\dt+\frac{\dx}{ \dt}\big )+T\eps.
\eeno
If we choose $\dt=\sqrt{\dx}$ and $\eps \leq C\sqrt{\dx}$ for $C\geq 0$ , we get the desired result.
\beno
  \|u^{n} - v^{n}\|_{\infty} \leq C \sqrt{ \dx},   
\eeno
}\fi

\subsection{Adding a limiter} \label{sec:limiter}
The basic filter scheme \eqref{eq:FS} is designed to be of high--order where the solution 
is regular and when there is no viscosity aspects.
However, for instance in the case of front propagation, 
it can be observed that the filter scheme may let small errors occur
near extrema, 
when two possible directions of propagation occur in the same cell.

This is the case for instance near a minima for an eikonal equation.
In order to improve the scheme near extrema, we propose to introduce a limiter
before doing the filtering process.
Limiting correction will be needed only when there is some viscosity aspect 
(it is not needed for advection).

Let us consider the case of front propagation,
i.e., equation of type \eqref{eq:hj}, now with
\be
   H(x,v_x)= \max_{a\in A} \big( f(x,a) v_x \big)
\ee
(i.e., no distributive cost in the Hamiltonian function).

In the one-dimensional case, a viscosity aspect may occur at a minima 
detected at mesh point $x_i$ if 
\begin{eqnarray}
\label{eq:changesigns}
  \mbox{ $\min_a f(x_j,a) \leq 0$ \ \  and \ \  $\max_a f(x_j,a)\geq 0$.}
\end{eqnarray}
In that case, the solution should not go below the local minima around this point,
i.e., we want 
\be\label{eq:limitmin}
   u^{n+1}_j \geq u_{min,j}:=\min(u^n_{j-1},u^n_j, u^n_{j+1}), 
\ee
and, in the same way, we want to impose that 
\be\label{eq:limitmax}
  u^{n+1}_j \leq u_{max,j}:=\max(u^n_{j-1},u^n_j, u^n_{j+1}).
\ee

If we consider the high-order scheme to be of the form 
$u^{n+1}_j=u^n_j - \dt h^A(u^n)$, 
then the limiting process amounts to saying that 
$$
  h^A(u^n)_j\leq  h^{max}_j:=\frac{u^n_j - u_{min,j}}{\dt}.
$$
and 
$$
  h^A(u^n)_j\geq  h^{min}_j:=\frac{u^n_j - u_{max,j}}{\dt}.
$$
This amounts to define a limited $\bar h^A$ such that, if \eqref{eq:changesigns} holds at mesh point $x_j$,
then 
$$
  \bar h^A(u^n)_j:=\min\bigg(\max(h^A(u^n)_j,\, h^{min}_j),\, h^{max}_j\bigg).
$$
and, otherwise,
$$
  \bar h^A_j :\equiv h^A_j.
$$ 
Then the filtering process is the same, using $\bar h^A$ instead of $h^A$ 
for the definition of the high-order scheme $S^F$.

For two dimensional equations a similar limiter could be developped in order to make the scheme more efficient 
at singular regions. However, for the numerical tests
of the next section (in two dimensions) we will simply limit the scheme by using an equivalent of 
\eqref{eq:limitmin}-\eqref{eq:limitmax}. Hence, instead of the scheme value $u^{n+1}_{ij}=S^A(u^n)_{ij}$
for the high--order scheme, we will update the value by
\be\label{eq:2dlimiter}
  u^{n+1}_{ij}=\min(\max(S^A(u^n)_{ij},u^{min}_{ij}),u^{max}_{ij}),
\ee
where
$u^{min}_{ij}=\min(u^n_{ij},u^n_{i\pm 1,j}, u^n_{i,j\pm 1})$ and
$u^{max}_{ij}=\max(u^n_{ij},u^n_{i\pm 1,j}, u^n_{i,j\pm 1})$.

\subsection{How to choose the parameter $\eps$: a simplified approach}
The scheme should switch to high--order scheme when some regularity of the data is detected, and in that case we should have
\beno
  \bigg| \frac{S^A(v)-S^M(v)}{\eps\tau} \bigg| \ = \ 
  \bigg| \frac{h^{A}(\cdot)-h^{M}(\cdot)}{\epsilon} \bigg|\ \leq\  1.
\eeno
In a region where a function $v=v(x)$ is regular enough, by using Taylor expansions, 
zero order terms in $h^A(x,D^\pm v)$ and $h^M(x,D^\pm v)$ vanish (they are both equal to $H(x,v_x(x))$)
and it remains an estimate of order $\dx$. More precisely, by using the high--order property (A4) we have 
$$h^A(x_j,D^\pm v_j)= H(x_j,v_x(x_j)) + O(\dx^2).$$
On the other hand, by using Taylor expansions,
$$Dv^\pm_j = v_x(x_j) \pm \frac{1}{2} v_{xx}(x_j) \dx + O(\dx^2),$$
Hence, denoting $h^M = h^M(x,u^-,u^+)$, it holds at points where $h^M$ is regular,
$$  
   h^M(x_j, Dv^-_j, Dv^+_j) = H(x_j,v_x(x_j)) + \frac{1}{2} v_{xx}(x_j) \bigg( 
   \frac{\partial h^M_j}{\partial u^+} - 
   \frac{\partial h^M_j}{\partial u^-} 
   \bigg)
   + O(\dx^2).
$$ 
Therefore,
\beno
  |h^{A}(v)-h^{M}(v)| 
  & = & 
  \frac{1}{2} |v_{xx}(x_j)|\, 
  \bigg| \frac{\partial h^M_j}{\partial u^+} - \frac{\partial h^M_j}{\partial u^-} \bigg|\,\dx + O(\dx^2).
\eeno
Hence we will make the choice to take $\eps$ roughly such that 
\be \label{eq:eps-rough-estimate}
 \frac{1}{2} |v_{xx}(x_j)|\,
   \bigg|   \frac{\partial h^M_j}{\partial u^+}
          - \frac{\partial h^M_j}{\partial u^-}
     \bigg|\,\dx \leq \eps
\ee
(where $h^M_j=h^M(x_j,v_x(x_j),v_x(x_j))$).
Therefore, if at some point $x_j$ \eqref{eq:eps-rough-estimate} holds, then the scheme will switch to the 
high-order scheme. Otherwise, when the expectations from $h^M$ and $h^A$ are different enough, the scheme will switch
to the monotone scheme.

In conclusion we have upper and lower bound for the switching parameter $\eps$:
\begin{itemize}
\item
Choose $\epsilon \leq  c_0 \sqrt{\Delta x}$ for some constant $c_0>0$ in order
that the convergence and error estimate result holds (see Theorem~\ref{th:estimate1}).

\item
Choose $\eps \geq c_1 \dx$, where $c_1$ is sufficiently large. This constant should be choosen roughly such 
that 
$$
  \frac{1}{2} \|v_{xx}\|_\infty 
  \bigg\| \frac{\partial h^M}{\partial u^+}(.,v_x,v_x) - \frac{\partial h^M}{\partial u^-}(.,v_x,v_x) \bigg\|_\infty \leq c_1.
$$
where the range of values of $v_x$ and $v_{xx}$ can be estimated, in general, from the values of $(v_0)_x$, $(v_0)_{xx}$
and the Hamiltonian function $H$. Then the scheme is expected to switch 
to the high-order scheme where the solution is regular.
\end{itemize}



\section{Numerical tests}

In this section we present several numerical tests in one and two dimensions. Unless otherwise indicated, the filtered scheme will refer to the scheme 
where the high-order Hamiltonian is the centered scheme in space (see Remark~\ref{rem:CFD-RK2}), with Heun (RK2) scheme
discretization in time (see in particular Eqs.~\eqref{eq:SA-RK2a}-\eqref{eq:SA-RK2b}).
Hereafter this scheme will be referred as the ``centered scheme".

The monotone finite difference scheme and function $h^M$ will be made precise for each example.

For the filtered scheme, unless otherwise precised, the switching coefficient 
$ \eps=5\dx.  $
will be used.
In practice we have numerically observed that taking 
$\eps=c_1\dx$ with $c_1$ sufficiently large does not much change the numerical results in the following tests.
All the tested filtered schemes (apart from the steady and obstacle equations)
enters in the convergence framework of the previous section,
so in particular there is a theoretical convergence of order $\sqrt{\dx}$ under the usual CFL condition.

In the tests, the filtered scheme will be in general compared to a second order ENO scheme 
(for precise definition, see Appendix~\ref{app:A}), as well as the centered (a priori unstable)
scheme without filtering.

In several cases, local error in the $L^2$ norms are computed in some subdomain $D$, which, at a given time $t_n$,
corresponds to
$$
  e_{L^2_{loc}}:= \bigg(\sum_{\{i,\ x_i\in D\}} |v(t_n,x_i) - u^n_i|^2\bigg)^{1/2}
$$

\medskip


The first two numerical examples deal with one-dimensional HJ equations,
examples 3 and 4 are concerned with two-dimensional HJ equations, and the last three examples will concern a one-dimensional 
steady equation and two nonlinear one-dimensional obstacle problems.


\medskip\noindent
{\bf Example 1.} {\bf Eikonal equation.}
We consider the case of 
\be\label{eq:eikonal}
  & & v_t + |v_x|=0, \quad  t\in(0,T), \ x\in (-2,2),\\
  & & v(0,x)=v_0(x):= \max(0,1-x^2)^4, \quad x\in(-2,2)\label{eq:v02a}. 
\ee

In Table~\ref{tab:ex2-1}, we compare the filtered scheme (with $\eps=5\dx$) with the centered scheme and 
the ENO second order scheme, with CFL $=0.37$ and terminal time $T=0.3$.
For the filtered scheme, the monotone hamiltonian used is $h^M(x,v^-,v^+):=max(v^-,-v^+)$. 

As expected, we observe that the centered scheme alone is unstable.
On the other hand, the filtered and ENO schemes are numerically comparable in that case, and second order convergent
(the results are similar for the $L^1$ and the $L^\infty$ errors).

Then, in Table~\ref{tab:ex2-2}, we consider the same PDE but with the following reversed initial data:
\be
  \label{eq:v02b}
  & & \tilde v_0(x):=-\max(0,1-x^2)^4, \quad x\in(-2,2).
\ee
In that case the centered scheme alone is unbounded. 
The filtered scheme (with $\eps=5\dx$) is second order.
However, here, the limiter correction as described in section \eqref{sec:limiter}
was needed in order to get second order behavior.
We also observe that the filtered scheme gives better results than the ENO scheme.
(We have also numerically tested the ENO scheme with the same limiter correction but it does not improve the behavior 
of the ENO scheme alone).

In conclusion, this first example shows firstly, that the filtered scheme can stabilize an otherwise unstable scheme,
and secondly that it can give the desired second order behavior.

\begin{table}[\tablepos]
\begin{center}
\begin{tabular}{|cc|cc|cc|cc|}
\hline
 &  
   & \multicolumn{2}{|c|}{filtered ($\eps= 5\dx$)} & \multicolumn{2}{|c|}{centered} & \multicolumn{2}{|c|}{ENO2} \\
\hline
  $M$ & $N$  & $L^2$ error   & order  & $L^2$ error    & order  &  $L^2$ error   & order  \\
\hline
\hline 
   40 &    9  & 7.51E-03 &   -    & 1.18E-01 &   -    & 1.64E-02 &   -     \\
   80 &   17  & 3.36E-03 &  1.16  & 1.14E-01 &  0.06  & 4.38E-03 &  1.91   \\
  160 &   33  & 8.02E-04 &  2.07  & 1.13E-01 &  0.00  & 1.19E-03 &  1.87   \\
  320 &   65  & 1.80E-04 &  2.16  & 1.13E-01 &  0.00  & 3.22E-04 &  1.89   \\
  640 &  130  & 4.53E-05 &  1.99  & 1.13E-01 &  0.00  & 8.22E-05 &  1.97   \\
\hline
\end{tabular}
\caption{(Example 1 with initial data \eqref{eq:v02a})
$L^2$ errors for filtered scheme, centered scheme, and ENO second order scheme 
\label{tab:ex2-1}
}
\end{center}
\end{table}

\begin{table}[\tablepos]
\begin{center}
\begin{tabular}{|cc|cc|cc|cc|}
\hline
  \multicolumn{2}{|c|}{ } 
   & \multicolumn{2}{|c|}{filtered ($\eps= 5\dx$)} & \multicolumn{2}{|c|}{centered} & \multicolumn{2}{|c|}{ENO2} \\
\hline
  $M$ & $N$  &  error   & order  & error    & order  &  error   & order  \\
\hline
\hline 

   40  &    9   & 1.27E-02 &   -        & 2.03E-02  &   -    & 2.60E-02 &        -     \\ 
   80  &   17  & 3.17E-03 &  2.00  & 8.96E-03  &  1.18  &  8.00E-03  & 1.70 \\ 
  160 &   33  & 7.90E-04 &  2.01  & 1.06E-02  & -0.24  & 2.50E-03  &  1.68 \\ 
  320 &   65  & 1.97E-04 &  2.00  & 1.26E-01  & -3.57  & 7.80E-04  &  1.68 \\ 
  640 &  130 & 4.92E-05 &  2.00  & 1.06E+02  & -9.71  & 2.44E-04 &  1.67 \\ 
\hline
\end{tabular}
\caption{(Example 1 with initial data \eqref{eq:v02b}.)
$L^2$ errors for filtered scheme, centered scheme, and ENO second order scheme.
\label{tab:ex2-2}
}
\end{center}
\end{table}


\newcommand{\wfact}{0.38}
\begin{figure}[!h]
\vspace{-1cm}
\hspace*{-0.5cm}
{\includegraphics[width=\wfact\linewidth]{\figsnew/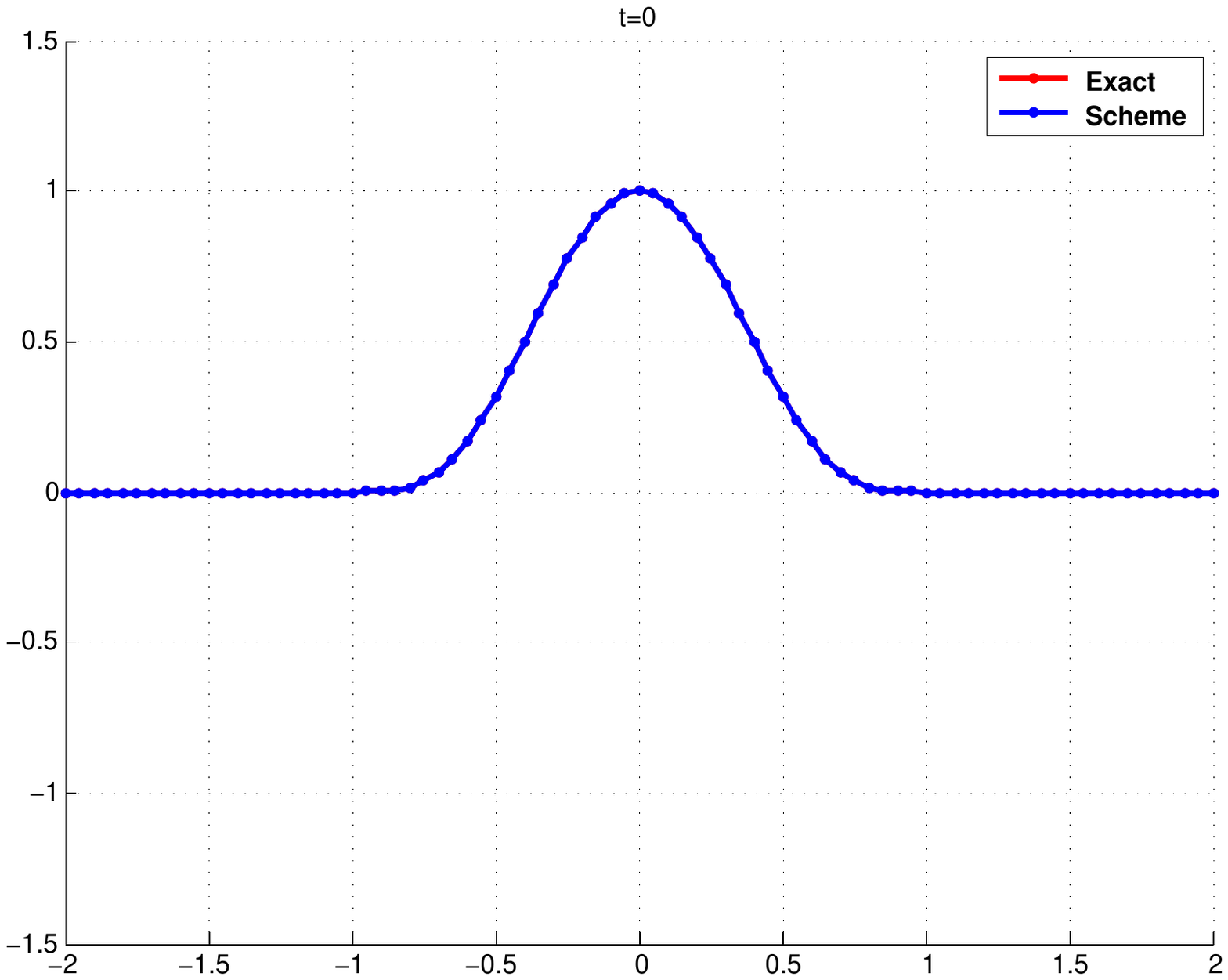}}\hspace*{-0.5cm}
{\includegraphics[width=\wfact\linewidth]{\figsnew/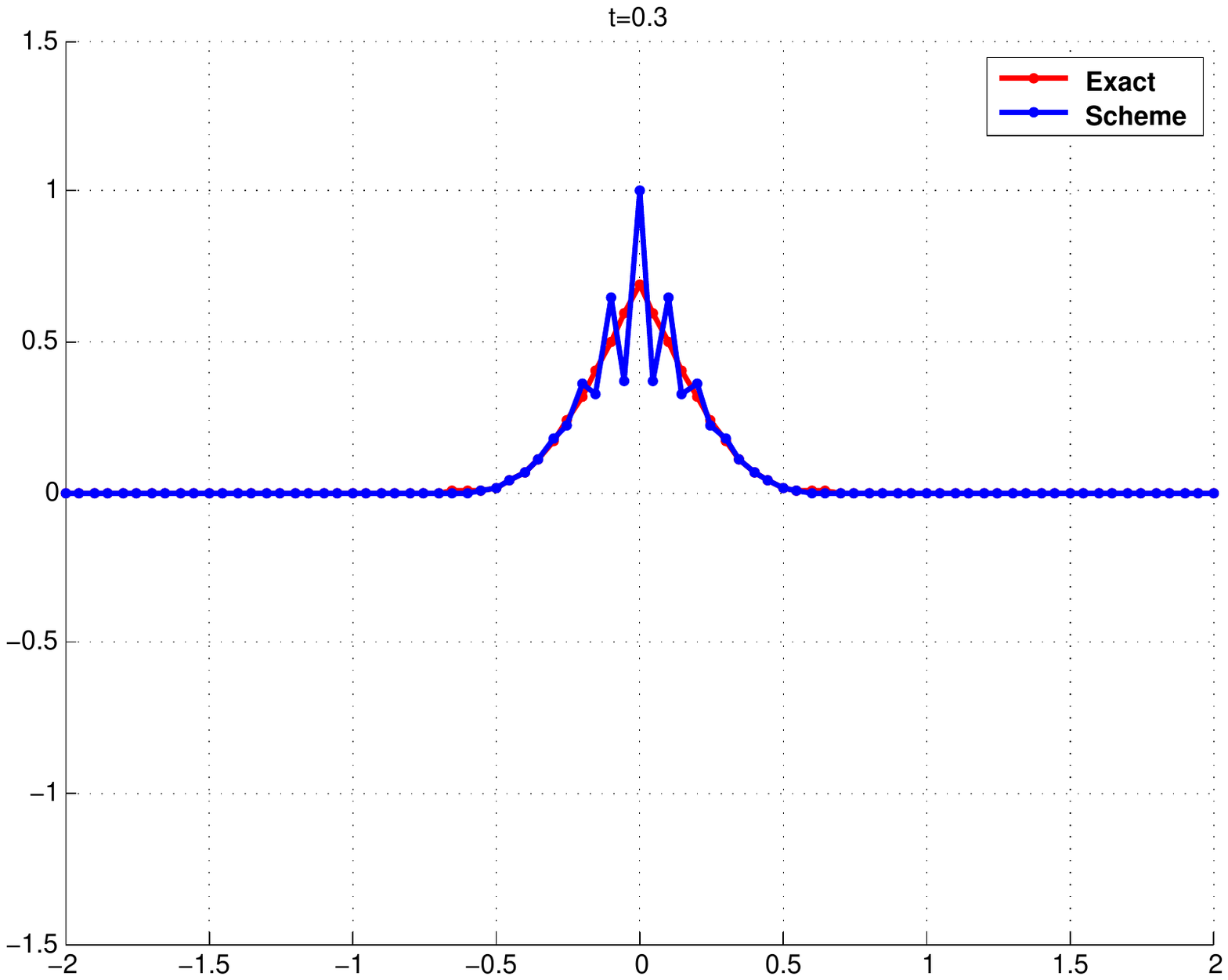}}\hspace*{-0.5cm}
{\includegraphics[width=\wfact\linewidth]{\figsnew/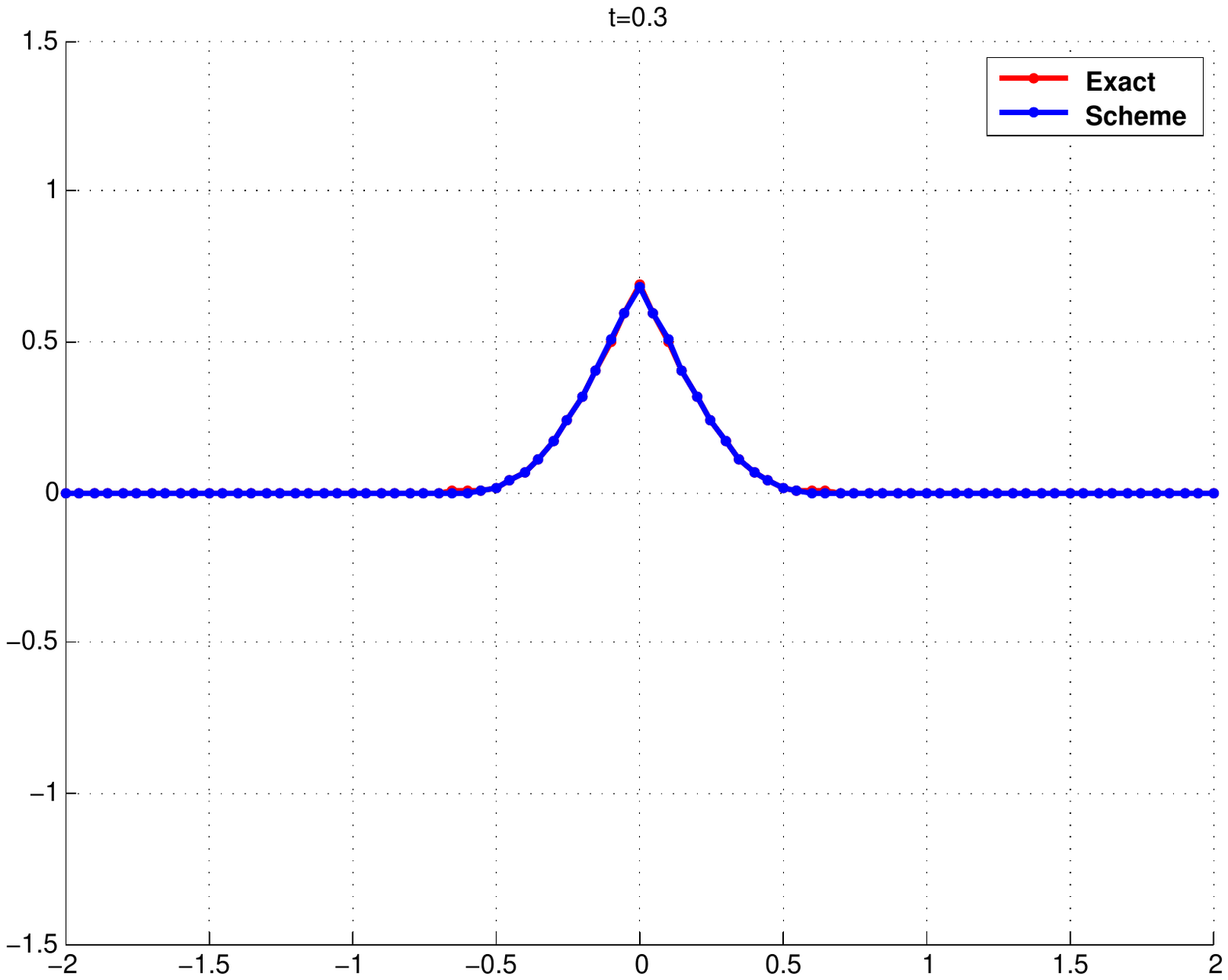}}\\[-1.5cm]
\vspace{-1cm}
\caption{(Example 1) With initial  data~\eqref{eq:v02a} (left), 
and plots at time $T=0.3$ with centered scheme - middle -  and filtered scheme - right, using $M=160$ mesh points.
\label{fig:v02a}
}
\end{figure}
\begin{figure}[!h]\hspace*{-0.5cm}
\includegraphics[width=\wfact\linewidth]{\figsnew/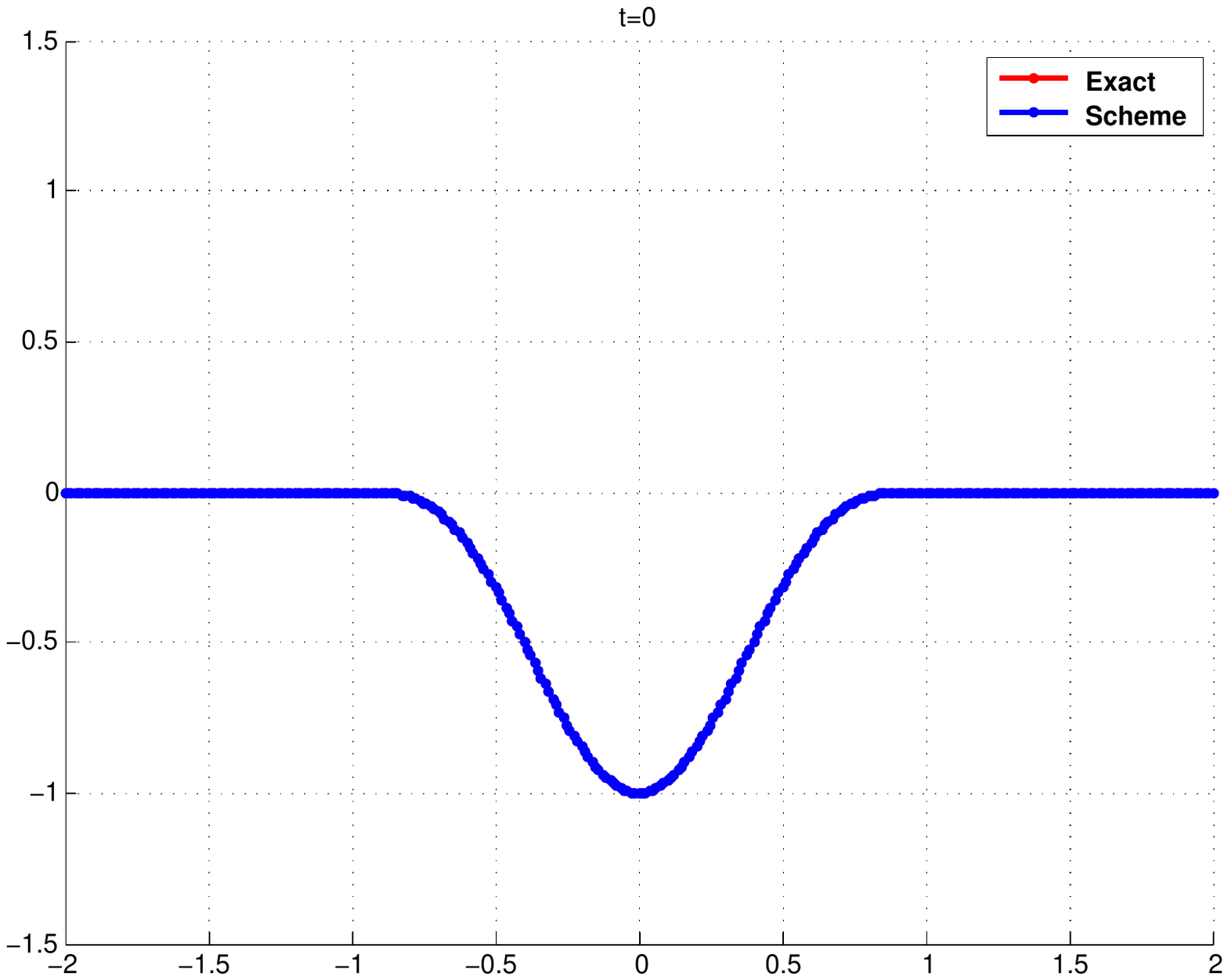}\hspace*{-0.5cm}
\includegraphics[width=\wfact\linewidth]{\figsnew/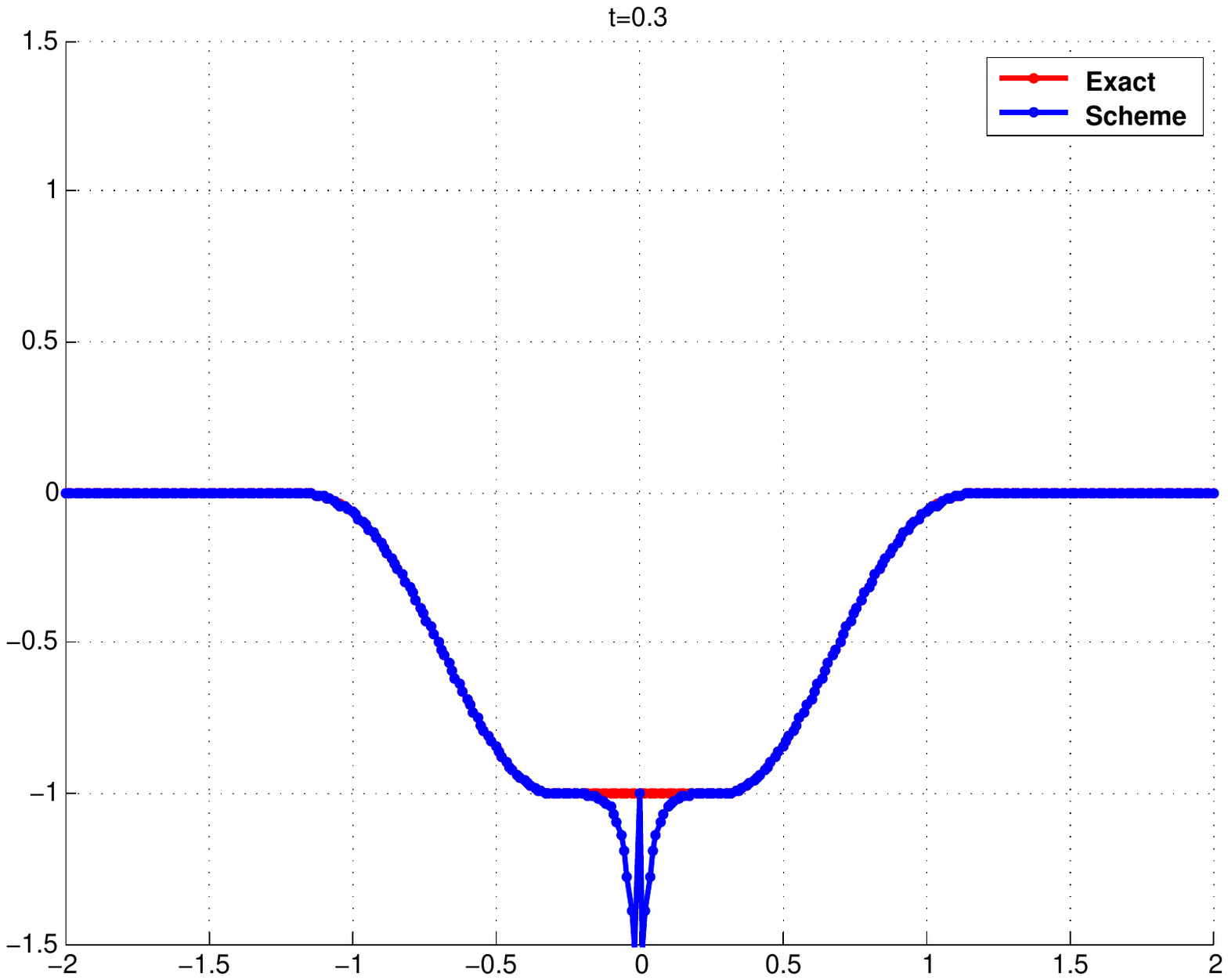}\hspace*{-0.5cm}
\includegraphics[width=\wfact\linewidth]{\figsnew/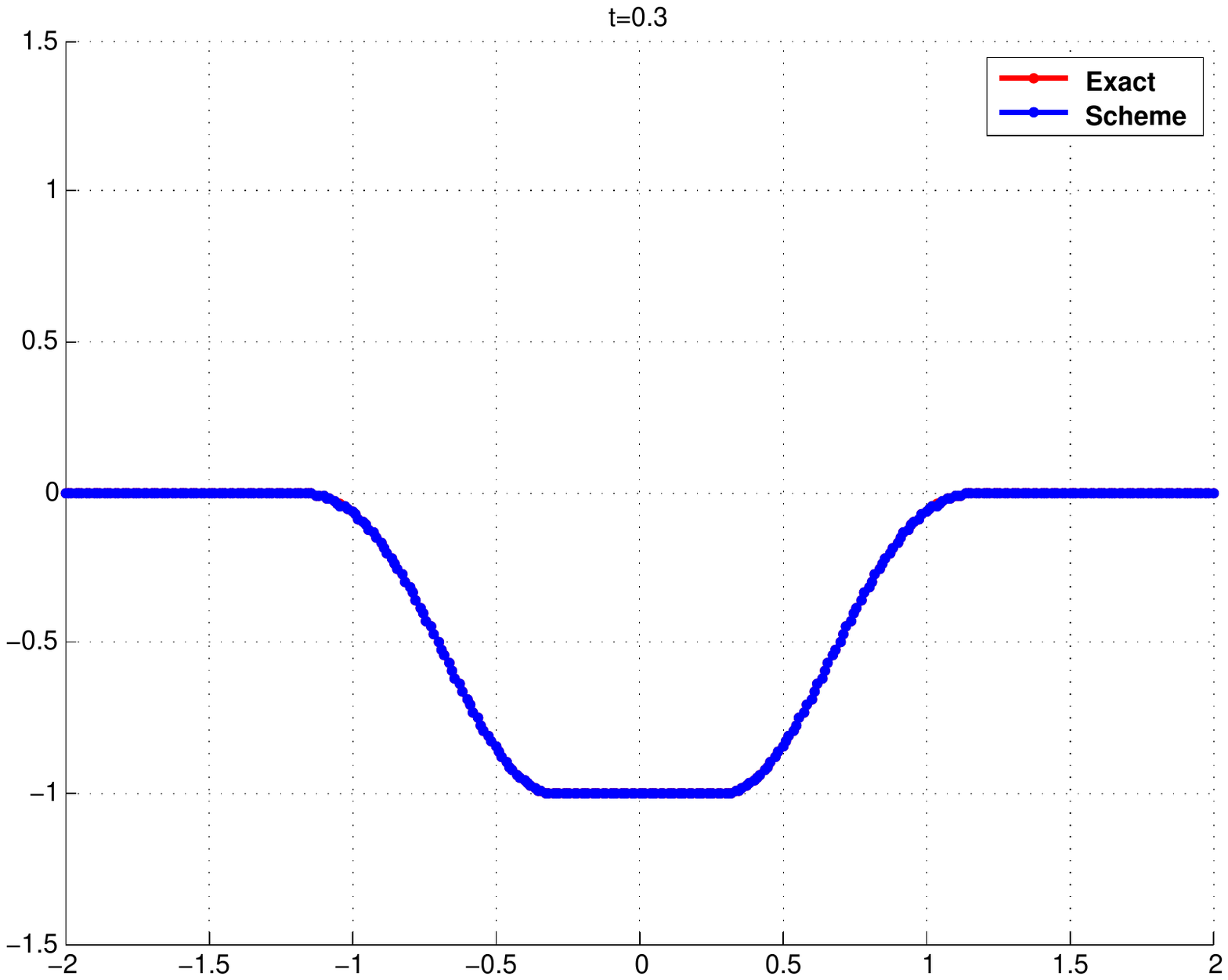}\\[-1.5cm]
\vspace{-1cm}
\caption{(Example 1) With initial  data~\eqref{eq:v02b} (left), 
and plots at time $T=0.3$ with centered scheme - middle -  and filtered scheme - right, using $M=160$ mesh points.
\label{fig:v02b}
}
\end{figure}

\medskip


\medskip\noindent
{\bf Example 2.} {\bf Burger's equation.}

In this example an HJ equivalent of the nonlinear Burger's equation is considered:
\begin{subequations}
\label{eq:ex3-1}
\be
 & & v_t + \displaystyle \ud |v_x|^2=0, \quad  t>0,\  x \in(-2,2)\\
 & &  v(0,x)=v_0(x):= \max(0,1-x^2), \quad x\in(-2,2)
\ee
\end{subequations}
with Dirichlet boundary condition on $(-2,2)$.
Exact solution is known.%
\footnote{
It holds 
$$
  v(t,x)=\frac{(\max(0,1-|\bar x|))^2}{2t} 1_{ \{t>\ud \} } + \frac{(1-2t)^2- |x|^2}{1-2t}\ 1_{ \{1\geq |x|\geq 1-2t \} }.
$$
}.
In order to test high--order convergence we have considered 
the smoother initial data which is the one obtained from \eqref{eq:ex3-1} at time 
$t_0:=0.1$,~i.e.~:
\begin{subequations}
\label{eq:ex3-2}
\be
 & & w_t+ \displaystyle \ud |w_x|^2=0, \quad  t>0,\  x \in(-2,2).\\
 & &  w(0,x):= v(t_0,x), \quad x\in(-2,2),
\ee
\end{subequations}
with exact solution $w(t,x)=v(t+t_0,x)$.

An illustration is given in Fig.~\ref{fig:burger}.
For the filtered scheme, the monotone hamiltonian used is 
$h^M(x,v^-,v^+):=\frac{1}{2}(v^-)^2\ 1_{v^->0} + \frac{1}{2}(v^+)^2\ 1_{v^+<0}$.

Errors are given in Table~\eqref{tab:ex3-1}, using  CFL=$0.37$ and terminal time $T=0.3$. 

In conclusion we observe numerically that the filtered scheme keeps the good behavior of the centered scheme
(here stable and almost second order).

\begin{figure}[!h]
\includegraphics[width=0.48\textwidth]{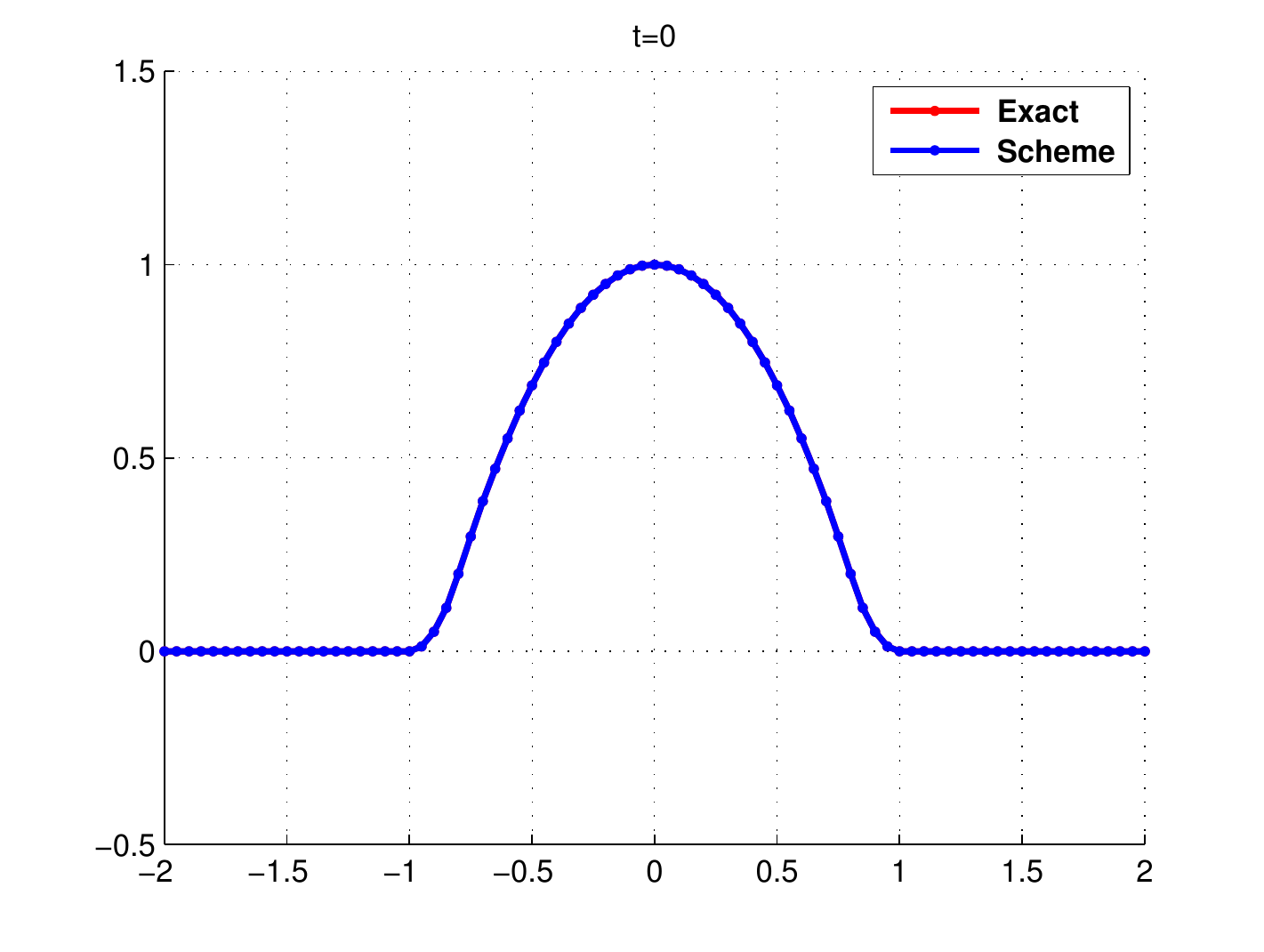} 
\includegraphics[width=0.48\textwidth]{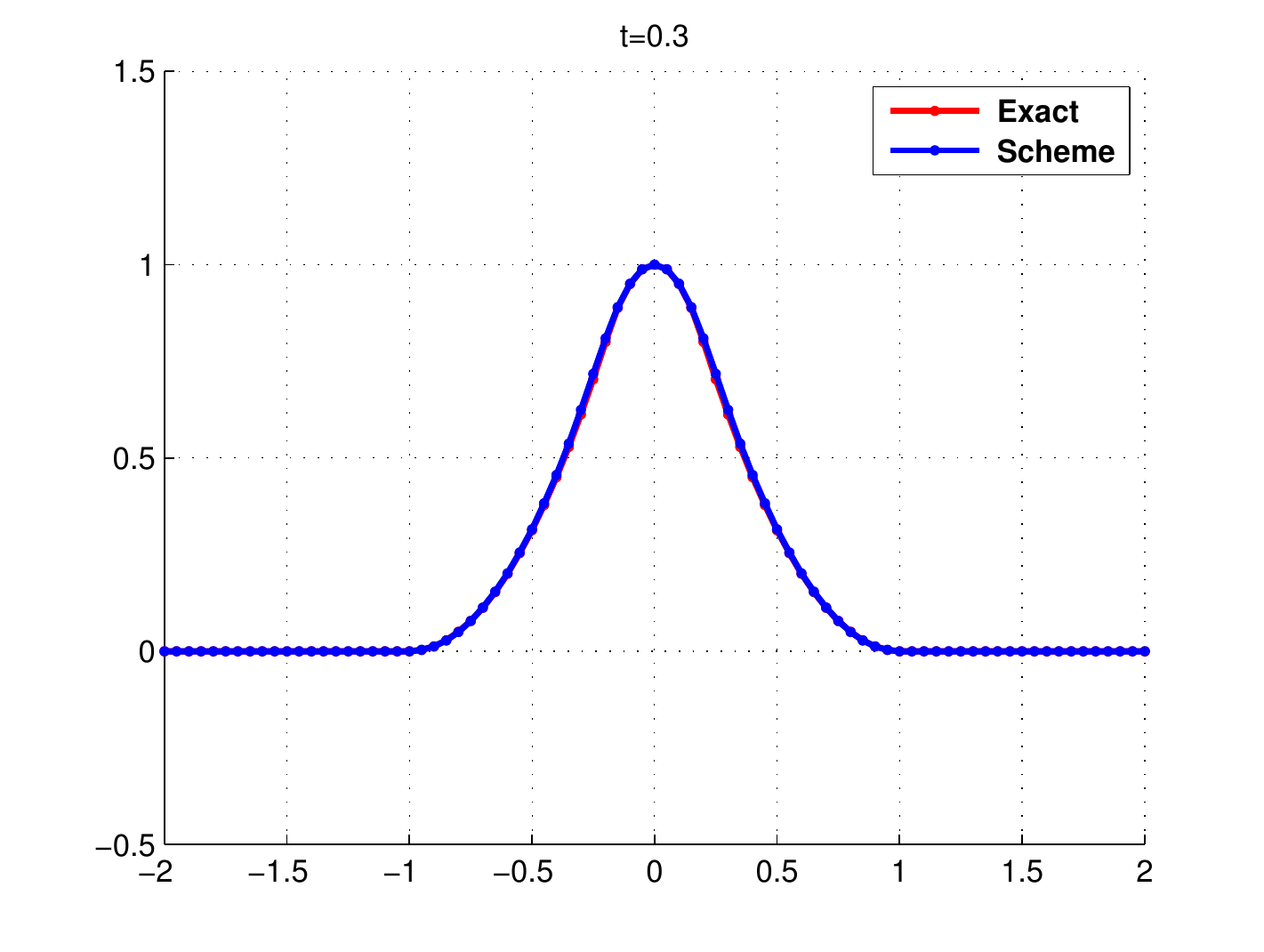}
\caption{(Example 2) Plots at $t=0$ and $t=0.3$ with the filtered scheme.
\label{fig:burger}}
\end{figure}

\begin{table}[\tablepos]
\begin{center}
\begin{tabular}{|cc|cc|cc|cc|}
\hline
  \multicolumn{2}{|c|}{ } 
   & \multicolumn{2}{|c|}{filtered ($\eps= 5\dx$)} & \multicolumn{2}{|c|}{centered} & \multicolumn{2}{|c|}{ENO2} \\
\hline
  $M$ & $N$  &  error   & order  & error    & order  &  error   & order  \\
\hline
\hline 

  40 &    9     & 2.06E-02  &   -        & 2.07E-02 &   -        & 2.55E-02 &   -  \\ 
   80 &   17   &  6.24E-03 &  1.73  & 6.24E-03 &  1.73  & 8.24E-03 &  1.63  \\ 
  160 &   33  & 1.85E-03  &  1.76  & 1.85E-03 &  1.76  & 2.81E-03 &  1.55  \\ 
  320 &   65  & 5.51E-04  &  1.74  & 5.51E-04 &  1.74  & 1.03E-03 &  1.45   \\ 
  640 &  130 & 1.68E-04  &  1.71  & 1.68E-04 &  1.71  & 3.74E-04 &  1.47  \\ 
\hline
\end{tabular}
\caption{(Example 2) $L^2$ errors for filtered scheme, centered scheme, and ENO second order scheme.
\label{tab:ex3-1}
}
\end{center}
\end{table}

\medskip\noindent
{\bf Example 3.} {\bf 2D rotation.}
We now apply filtered scheme to an advection equation in two dimensions:
\be
  & & v_t - y v_x + x v_y=0,\quad (x,y)\in \Omega,\ t>0,\\
  & & v(0,x,y)=v_0(x,y):= 0.5- 0.5\, \max(0,\frac{1-(x-1)^2-y^2}{1- r_0^2})^4
\ee
where $\Omega:=(-A,A)^2$ (with $A=2.5$),
$r_0=0.5$ and with Dirichlet boundary condition $v(t,x)=0.5$, $x\in\partial\Omega$.
This initial condition is regular and such that the level set $v_0(x,y)=0$
corresponds to a circle centered at $(1,0)$ and of radius $r_0$.

In this example the monotone numerical Hamiltonian is defined by
\be\label{eq:FD2}
  h^{M}(u^-_x,u_x^+,u_y^-,u_y^+)
    &:= & \max(0,f_1(a,x,y))u_x^- + \min(0,f_1(a,x,y))u_x^+    \\
    &  &  +\ \max(0,f_2(a,x,y))u_y^-+ \min(0,f_2(a,x,y))u_y^+ \nonumber
\ee
and the high--order scheme is the centered finite difference scheme in both spacial variables, and RK2 in time. 
The filtered scheme is otherwise the same as \eqref{eq:FS}. However it is necessary to use a greater constant $c_1$
is the choice $\eps=c_1\dx$ in order to get (global) second order errors.
We have used here $\eps=20\dx$. 

On the other hand the CFL condition is
\be\label{eq:CFL2}
  \mu:=c_0(\frac{\dt}{\dx}+\frac{\dt}{\dy} )\leq 1,
\ee
where here $c_0=2.5$ (an upper bound for the dynamics in the considered domain $\Omega$).
In this test the CFL number is $\mu:=0.37$. 

Results are shown in Table~\ref{tab:ex3-1-2D}
for terminal time time $T:=\pi/2$.
Although the centered scheme is a priori unstable, in this example it is numerically stable and of second order.
So we observe that the filtered scheme keep this good behavior and is also 
or second order (ENO scheme gives comparable results here).

\begin{table}[\tablepos]
\begin{center}
\begin{tabular}{|cc|cc|cc|cc|}
\hline
   &   & \multicolumn{2}{|c|}{filtered} & \multicolumn{2}{|c|}{centered} & \multicolumn{2}{|c|}{ENO} \\
\hline
  $Mx$ & $Ny$  &  $L^2$ error   & order  & $L^2$ error    & order  &  $L^2$ error   & order  \\
\hline
\hline 
   20 &   20  & 5.05E-01 &   -   & 5.05E-01 &   -   & 6.99E-01 &   -    \\
   40 &   40  & 1.48E-01 &  1.77 & 1.48E-01 &  1.77 & 4.66E-01 &  0.58  \\
   80 &   80  & 3.77E-02 &  1.98 & 3.77E-02 &  1.98 & 2.04E-01 &  1.19  \\ 
  160 &  160  & 9.40E-03 &  2.00 & 9.40E-03 &  2.00 & 5.50E-02 &  1.89  \\
  320 &  320  & 2.34E-03 &  2.01 & 2.34E-03 &  2.01 & 1.29E-02 &  2.10  \\
\hline
\end{tabular}
\caption{(Example 3) Global $L^2$ errors for the filtered scheme, centered and second order ENO schemes
(with CFL $0.37$).
\label{tab:ex3-1-2D}
}
\end{center}
\end{table}


\if{
\begin{figure}[!h]
\includegraphics[width=0.6\textwidth]{figs/advection}
\includegraphics[width=0.6\textwidth]{figs/advectionfs}\\
Initial  data (left), and plots at time $T=\pi/2$, by filtered scheme  ($M=40$ mesh points).
\end{figure}
}\fi

\begin{figure}[!h]
\includegraphics[width=0.45\textwidth]{\figsnew/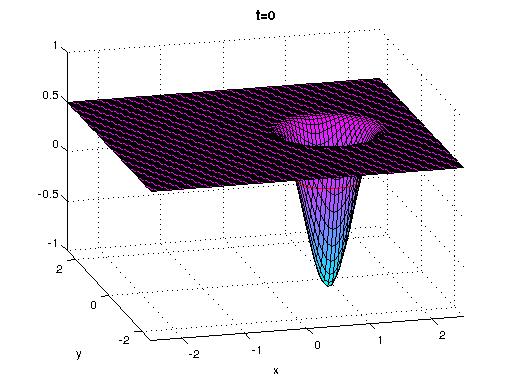}  
\includegraphics[width=0.45\textwidth]{\figsnew/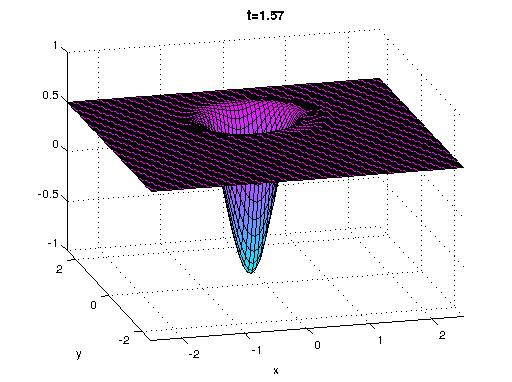}
\caption{\label{fig:ex3}
(Example 3)
Filtered scheme, plots at time  $t=0$ (left) and $t=\pi/2$ (rigth) with $M=80$ mesh points.
}
\end{figure}

\medskip\noindent
{\bf Example 4.} {\bf{Eikonal equation.}} 
In this example we consider the eikonal equation
\be\label{eq:eikonal2d}
  & & v_t+ |\nabla v|=0,\quad (x,y)\in \Omega,\ t>0
\ee
in the domain $\Omega:=(-3,3)^2$. 
The initial data is defined by 
{\small
\be \label{eq:eikonal-twoholes}
  & & \hspace{-1cm} v_0(x,y) = \\
  & & 0.5-0.5\, \max\bigg( \max(0,\frac{1-(x-1)^2-y^2}{1- r_0^2})^4,\ \max(0,\frac{1-(x+1)^2-y^2}{1- r_0^2})^4\bigg).
  \nonumber 
\ee
}
The zero-level set of $v_0$ corresponds to two separates circles or radius $r_0$ and 
centered in $A=(1,0)$ and $B=(-1,0)$ respectively.
Dirchlet boundary conditions are used as the previous example.

\if{
We first consider the following smooth initial data 
\be \label{eq:eikonal-onehole}
  & & v(0,x,y)=v_0(x,y)=0.5 -0.5\, \max(0,\frac{1-x^2-y^2}{1- r_0^2})^4,
\ee
with $r_0=0.5$, with Drichlet boundary conditions as in the previous example.

Numerical results are given in Table~\ref{tab:ex41} were is compared the global $L^2$ errors for the
filtered scheme (with $\eps=20\ex$), the centered scheme, and  a second order ENO scheme.
}\fi

The monotone hamiltonian $h^M$ used in the filtered scheme is in Lax-Friedrichs form:
\be\label{eq:LF2}
  h^{M}(x,u^-_1,u_1^+,u_2^-,u_2^+)
    & = & H(x,\frac{u^-_1+u_1^+}{2},\frac{u_2^-+u_2^+}{2}) \nonumber \\
    &   &  \quad  - \frac{C_x}{2}(u_1^+ - u_1^-) - \frac{C_y}{2}(u_2^+ - u_2^-),
\ee 
where, here, $C_x=C_y=1$.
We used the CFL condition $\mu=0.37$ as in \eqref{eq:CFL2}.
Also, the simple limiter~\eqref{eq:2dlimiter} was used for the filtered scheme as described in Section~\ref{sec:limiter}.
It is needed in order to get a good second order behavior at extrema of the numerical solution.
The filter coefficient is set to $\eps=20\dx$ as in the previous example.

\if{
\begin{table}
\begin{center}
\begin{tabular}{|cc|cc|cc|cc|}
\hline
   &   & \multicolumn{2}{|c|}{filtered ($\eps=20\dx$)} & \multicolumn{2}{|c|}{centered} & \multicolumn{2}{|c|}{ENO2} \\
\hline
  $Mx$ & $Ny$  &  $L^2$ error   & order  & $L^2$ error    & order  &  $L^2$ error   & order  \\
\hline
\hline 
   25 &   25  & 5.39E-01 &   -   & 3.73E-01 &   -    & 4.22E-01 &   -    \\
   50 &   50  & 1.82E-01 &  1.57 & 1.42E-01 &  1.39  & 1.57E-01 &  1.42  \\
  100 &  100  & 3.72E-02 &  2.29 & 4.72E-02 &  1.59  & 5.12E-02 &  1.62  \\
  200 &  200  & 9.36E-03 &  1.99 & 1.66E-02 &  1.51  & 1.48E-02 &  1.80  \\
  400 &  400  & 2.36E-03 &  1.99 & 7.23E-03 &  1.20  & 4.34E-03 &  1.77  \\
\hline
\end{tabular}
\caption{(Example 4, with initial data \eqref{eq:eikonal-onehole}) 
Global $L^2$ errors for filtered scheme, centered and second order ENO schemes, 
at time $t=0.6$.
\label{tab:ex41}
}
\end{center}
\end{table}

\begin{figure}
\includegraphics[width=0.5\textwidth]{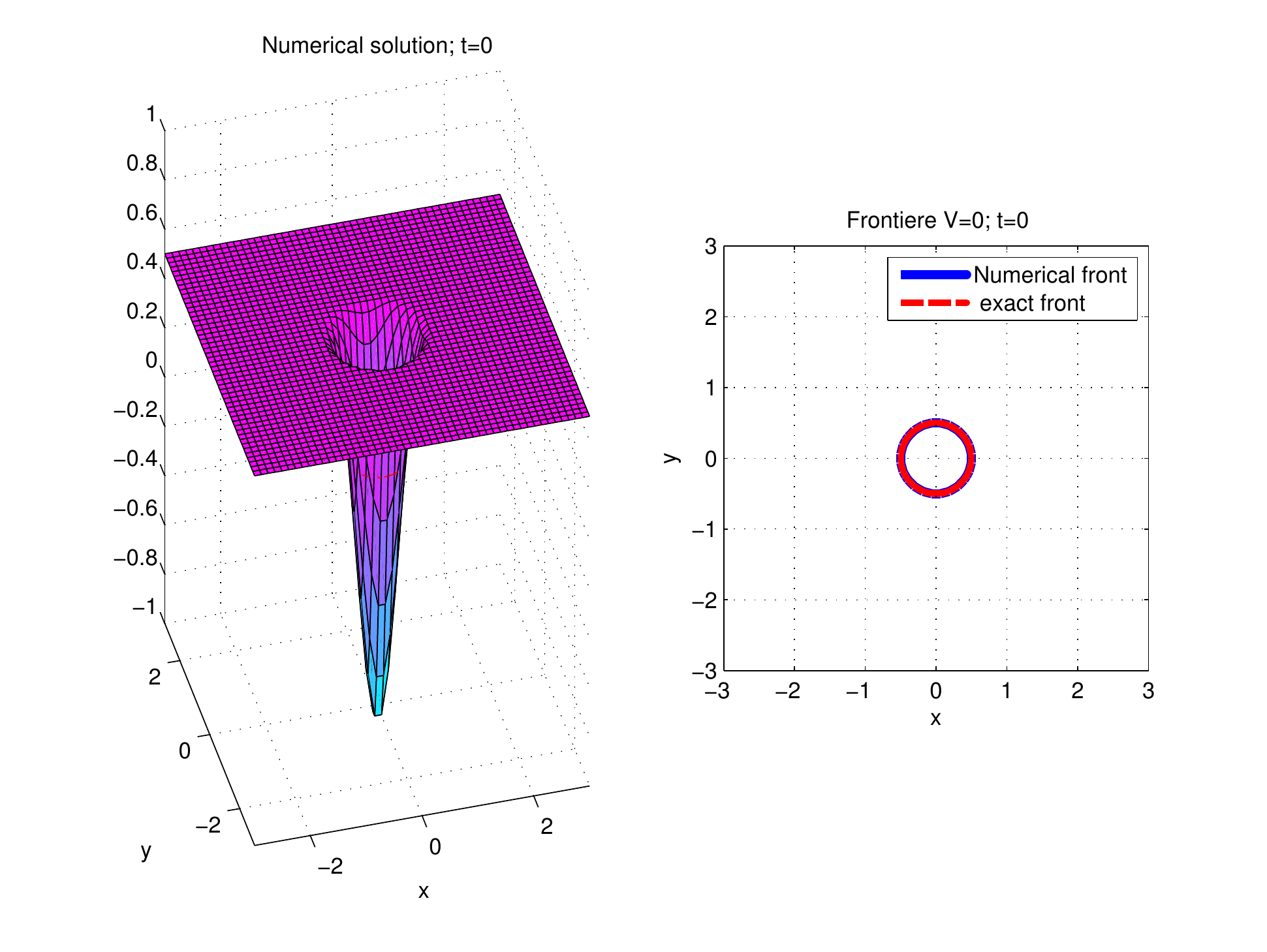}\\
\includegraphics[width=0.5\textwidth]{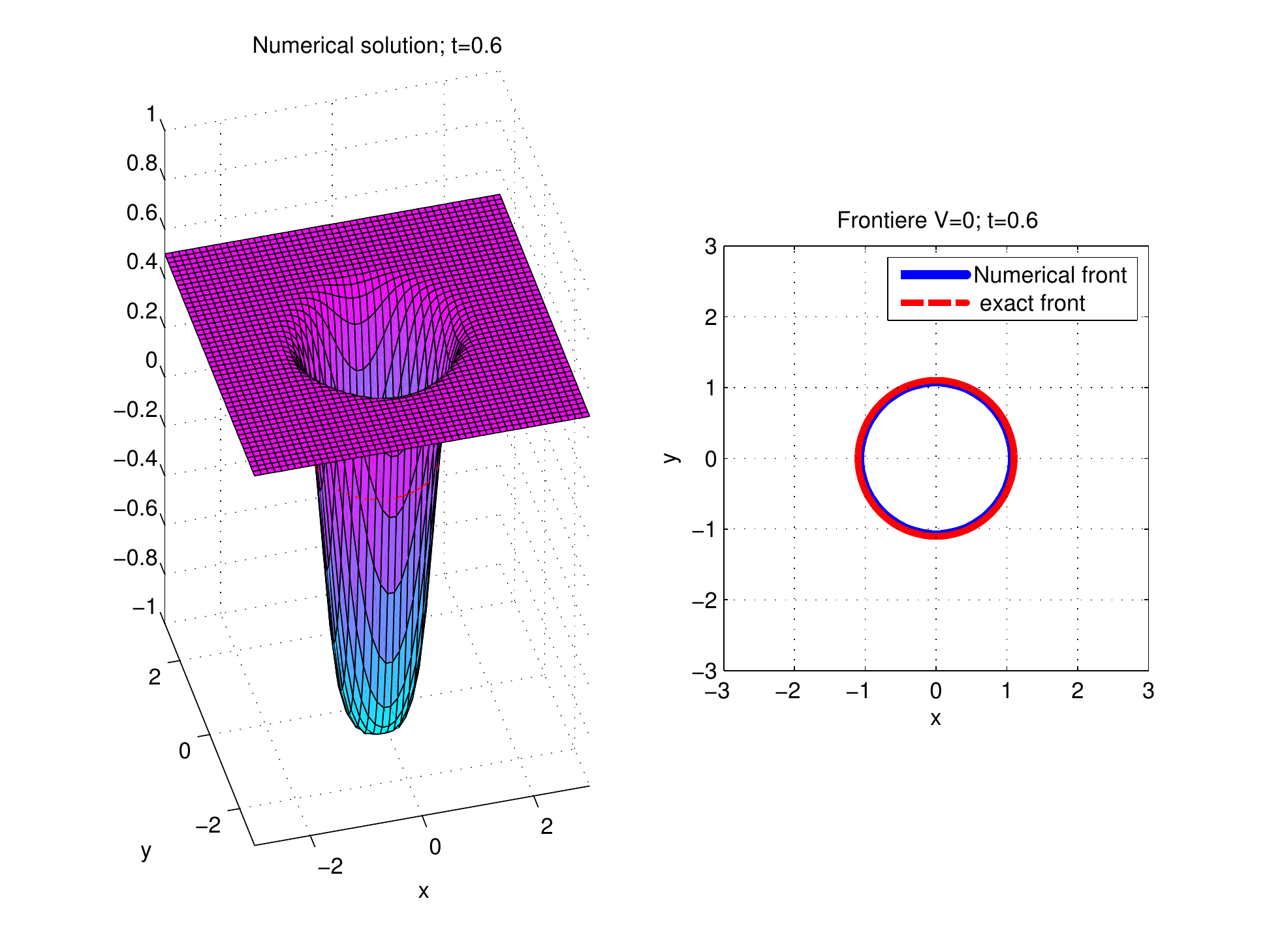}\\
Initial  data (Top), and plots at time $T=\pi/2$, by filtered scheme ($M=50$ mesh points).

\end{figure}
}\fi

\if{
We now consider an other initial data corresponding to two separates holes centered in $A=(1,0)$ and $B=(-1,0)$
respectively:
{\small
\be \label{eq:eikonal-twoholes}
  & & \hspace{-1cm} v(0,x,y) = \\
  & & 0.5-0.5\, \max\bigg( \max(0,\frac{1-(x-1)^2-y^2}{1- r_0^2})^4,\ \max(0,\frac{1-(x+1)^2-y^2}{1- r_0^2})^4\bigg).
  \nonumber 
\ee
}
}\fi

Numerical results are given in Table~\ref{tab:ex42}, showing the global $L^2$ errors for the
filtered scheme, the centered scheme, and a second order ENO scheme, at time $t=0.6$.
We observe that the centered scheme has some unstabilities for fine mesh, while the filtered performs as 
expected.

\begin{table}[!hbtp]
\begin{center}
\begin{tabular}{|cc|cc|cc|cc|}
\hline
   &   & \multicolumn{2}{|c|}{filtered ($\eps=20\dx$)} & \multicolumn{2}{|c|}{centered} & \multicolumn{2}{|c|}{ENO2} \\
\hline
  $Mx$ & $Ny$  &  $L^2$ error   & order  & $L^2$ error    & order  &  $L^2$ error   & order  \\
\hline
\hline
   25 &   25  & 5.39E-01 &   -   & 6.00E-01 &   -   & 5.84E-01 &   -   \\
   50 &   50  & 1.82E-01 &  1.57 & 2.25E-01 &  1.41 & 2.11E-01 &  1.47 \\
  100 &  100  & 3.72E-02 &  2.29 & 8.46E-02 &  1.41 & 6.88E-02 &  1.62 \\
  200 &  200  & 9.36E-03 &  1.99 & 3.53E-02 &  1.26 & 2.02E-02 &  1.76 \\
  400 &  400  & 2.36E-03 &  1.99 & 1.36E-01 & -1.95 & 5.98E-03 &  1.76 \\
\hline
\end{tabular}
\caption{(Example 4)
Global $L^2$ errors for filtered scheme, centered and  second order ENO schemes.
\label{tab:ex42}
}
\end{center}
\end{table}

\begin{figure}[!h]
\begin{center}
\begin{tabular}{cc}
\includegraphics[width=0.4\textwidth]{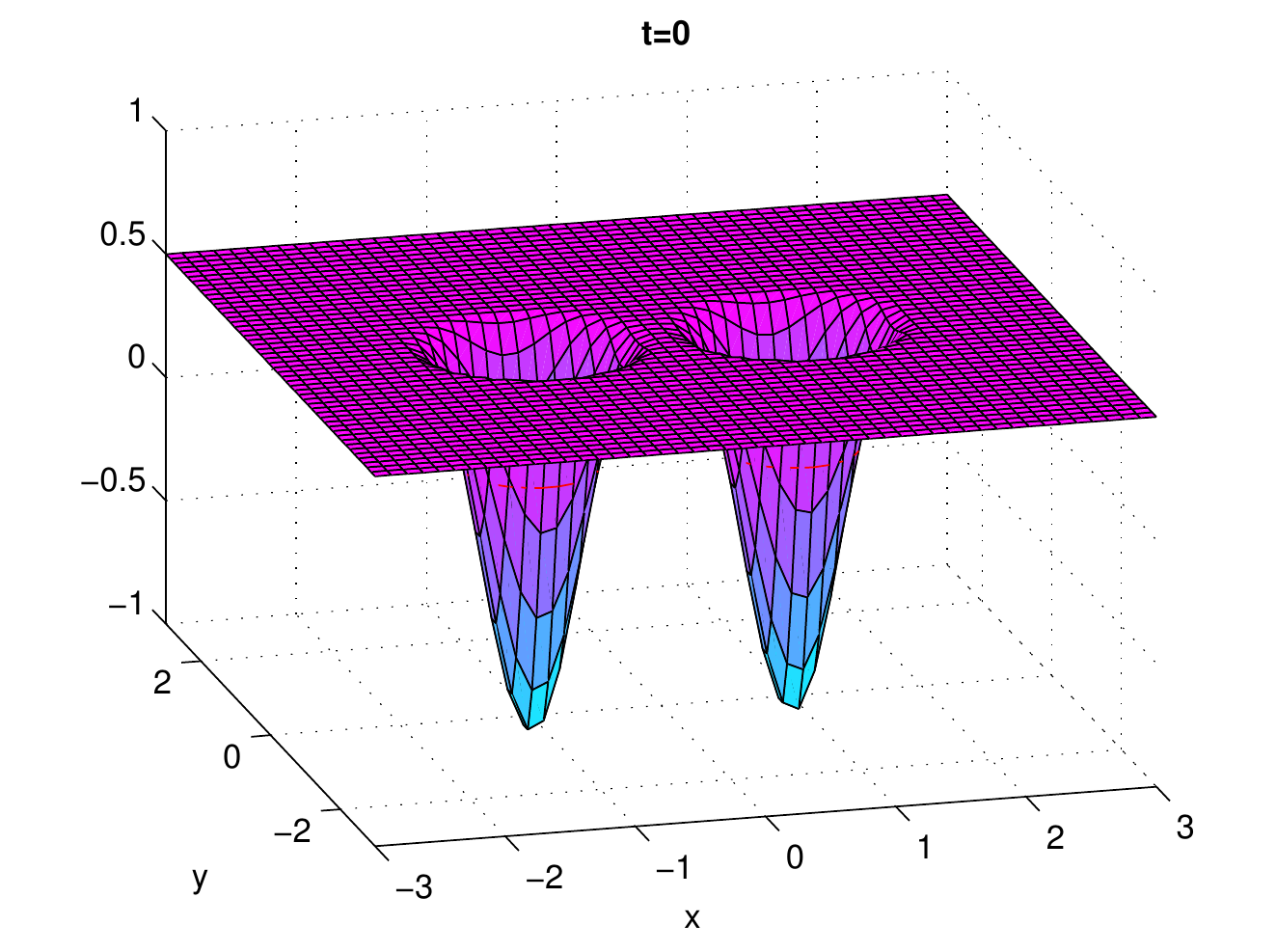} &
\includegraphics[width=0.4\textwidth]{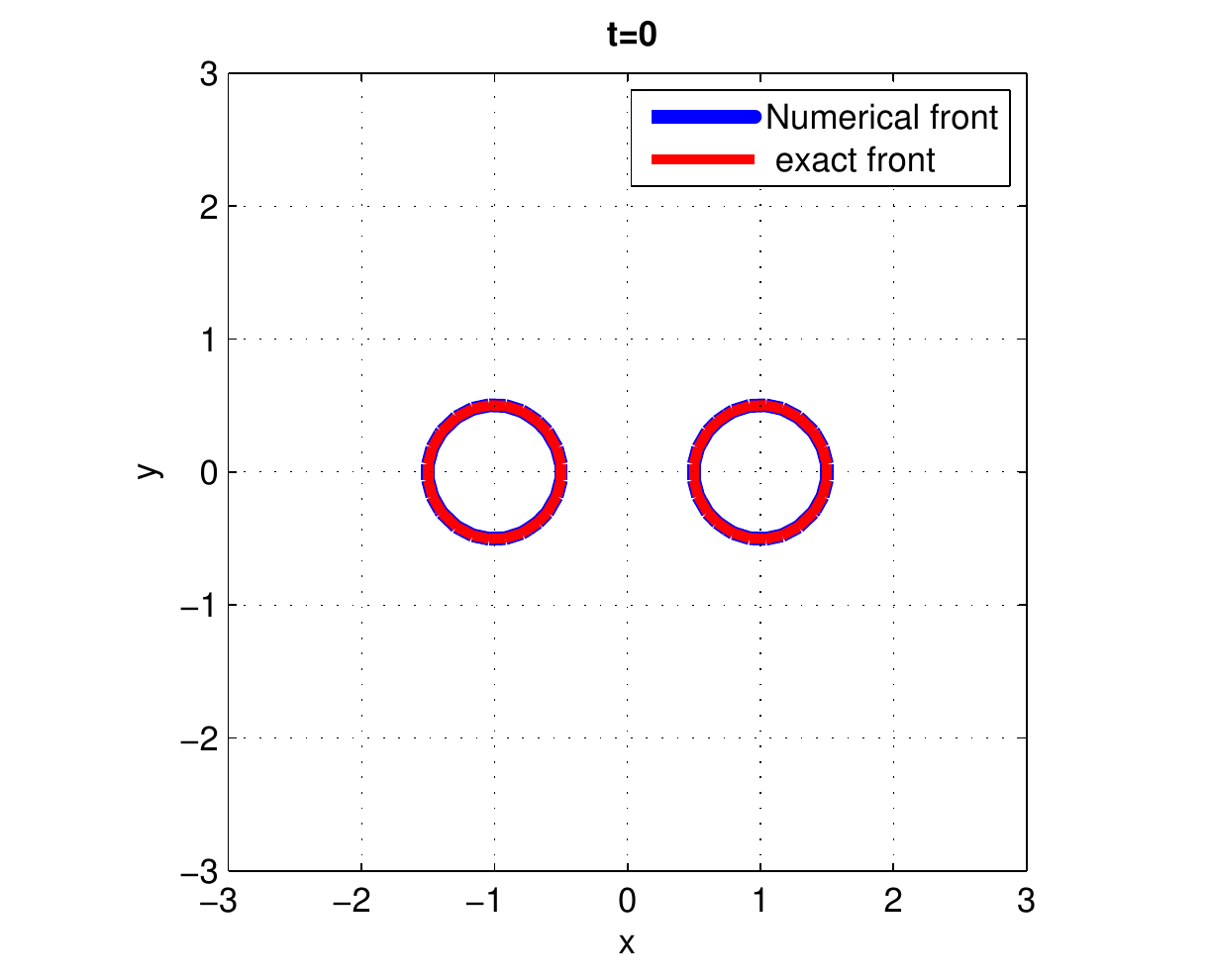}  \\
\includegraphics[width=0.4\textwidth]{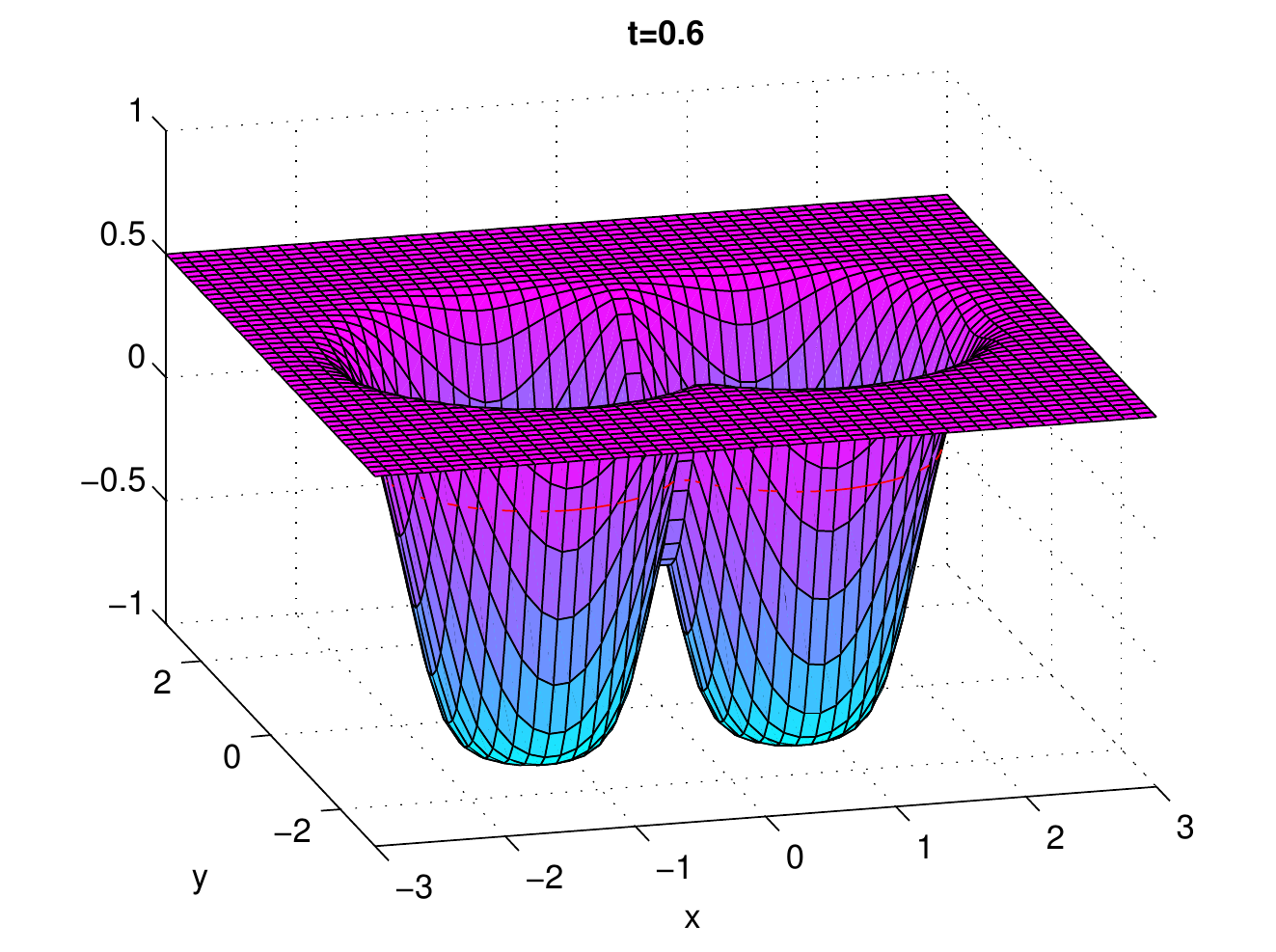} &
\includegraphics[width=0.4\textwidth]{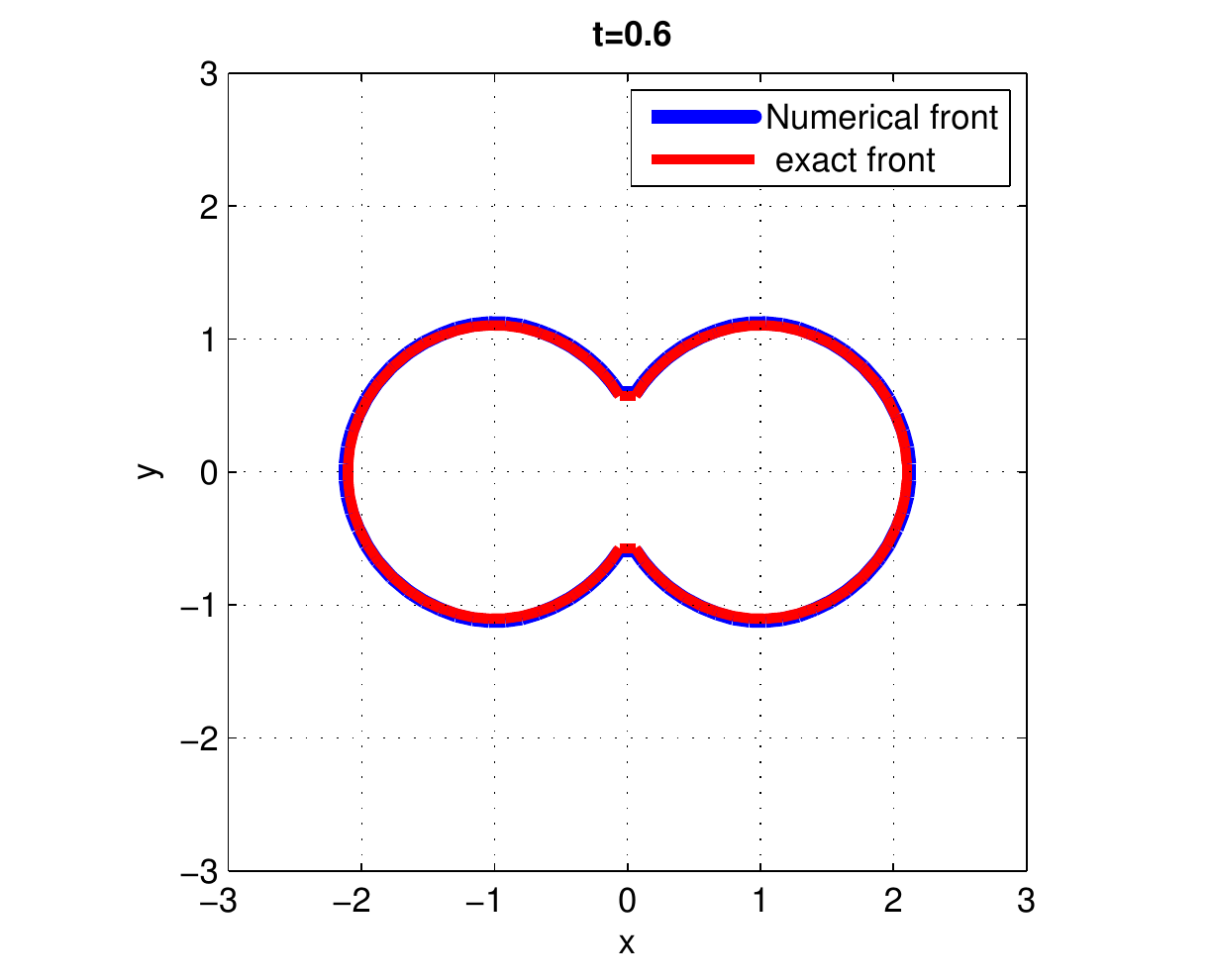}  
\end{tabular}
\end{center}
\caption{\label{fig:ex42}
(Example 4)
Plots at times $t=0$ (top) and $t=\pi/2$ (bottom)
for the filtered scheme with $M=50$ mesh points. The figures to the right represent the $0$-level sets.
}
\end{figure}

\if{
\medskip\noindent
{\bf Example 5.} In this example the considered HJ equation is 
\be
  v_t-yv_x+xv_y+ \|\nabla v\|=0, \quad (x,y)\in \Omega,\ t>0,
\ee
with $\Omega=(-3,3)^2$, and with the following initial data:
\be
   v_0(x,y)=\min(0.5,\|x-A\|_{2}-0.5,\|x- B\|_{2}-0.5),
\ee
where $\|.\|$ is Euclidean norm, $A=(1,0)$ and $B=(-1,0)$
(together with Dirichlet boundary condition $v(t,x,y)=0.5$ for $(x,y)\in\partial \Omega$).

Again we compare the filtered scheme 
(with $\eps=5\dx$) with the centered (a priori unstable) scheme and the second order ENO scheme. 

Numerical results are shown in Table~\ref{tab:ex5},
for terminal time $T=0.75$ and CFL $0.37$.
Local errors has been computed in the region $|v(t,x,y)| \leq 0.1$ and also away from the singular 
line $x+y=0$ (i.e., for points such that furthermore  $|\frac{x+y}{\sqrt{2}}|\geq 0.1$). 
In this example, the naive centered scheme is unstable (as expected), while the filtered scheme 
is stable and of second order.

\begin{table}[H]
\begin{center}
\begin{tabular}{|cc|cc|cc|cc|}
\hline
   &   & \multicolumn{2}{|c|}{filtered $\eps=5\dx$} & \multicolumn{2}{|c|}{centered} & \multicolumn{2}{|c|}{ENO2} \\
\hline
  $Mx$ & $Nx$  &  $L^2$ error   & order  & $L^2$ error    & order  &  $L^2$ error   & order  \\
\hline
\hline
%
%
%
  25 &   25 & 1.02E-01 &  0.00  & 1.11E-01 &  0.00  & 1.14E-01 &  0.00 \\ 
  50 &   50 & 1.78E-02 &  2.51  & 1.99E-02 &  2.48  & 2.12E-02 &  2.43 \\ 
 100 &  100 & 6.06E-03 &  1.56  & 2.04E-02 & -0.03  & 3.67E-03 &  2.53 \\ 
 200 &  200 & 1.13E-03 &  2.43  & 1.27E-02 &  0.69  & 8.61E-04 &  2.09 \\ 
 400 &  400 & 2.86E-04 &  1.98  & 1.13E-02 &  0.17  & 2.12E-04 &  2.02 \\ 
\hline
\end{tabular}
\caption{(Example 5) Local errors of filtered, centered and ENO scheme.
\label{tab:ex5}
}
\end{center}
\end{table}

\begin{figure}[!h]
\includegraphics[width=0.4\textwidth]{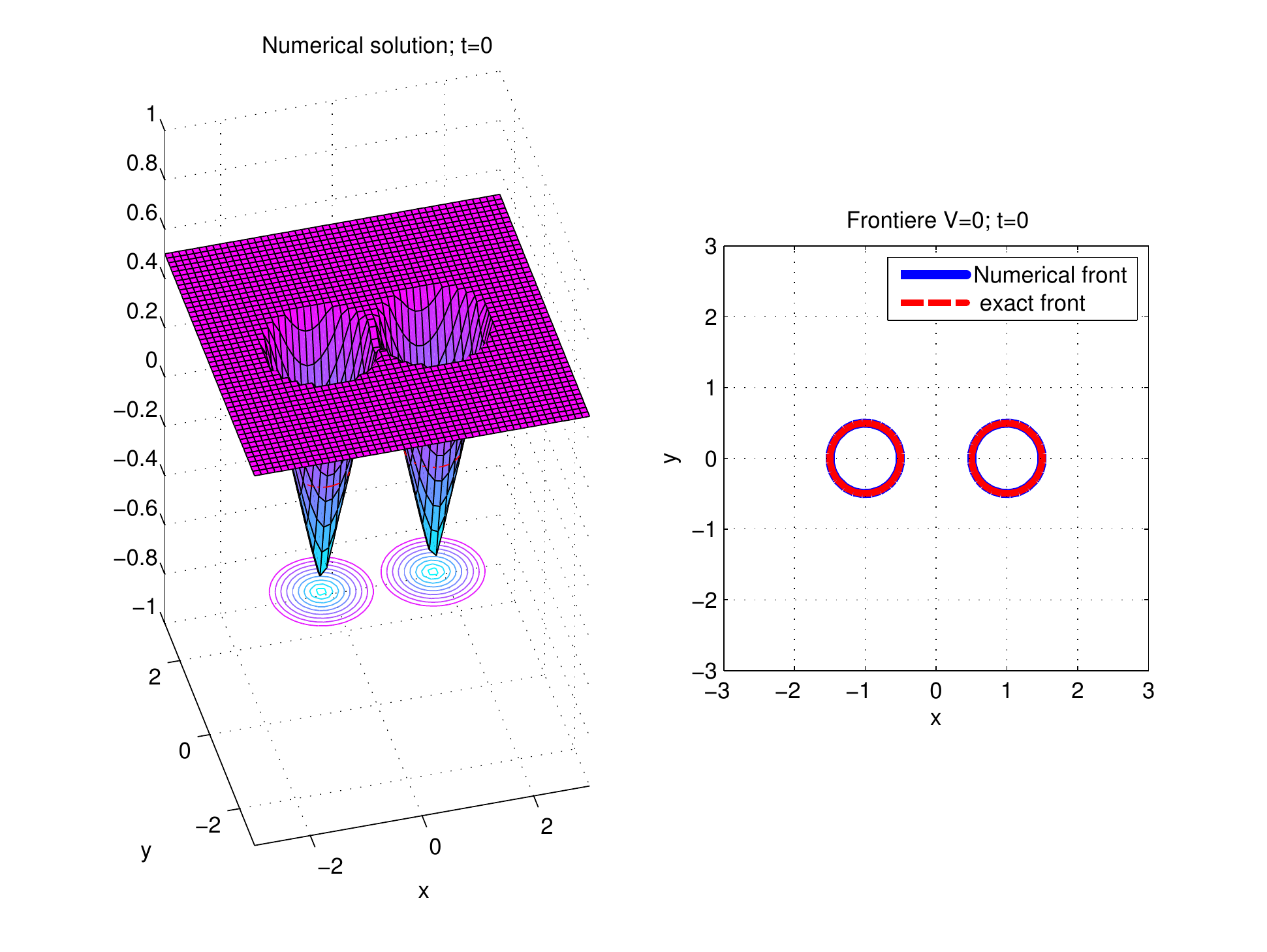}\\
\includegraphics[width=0.4\textwidth]{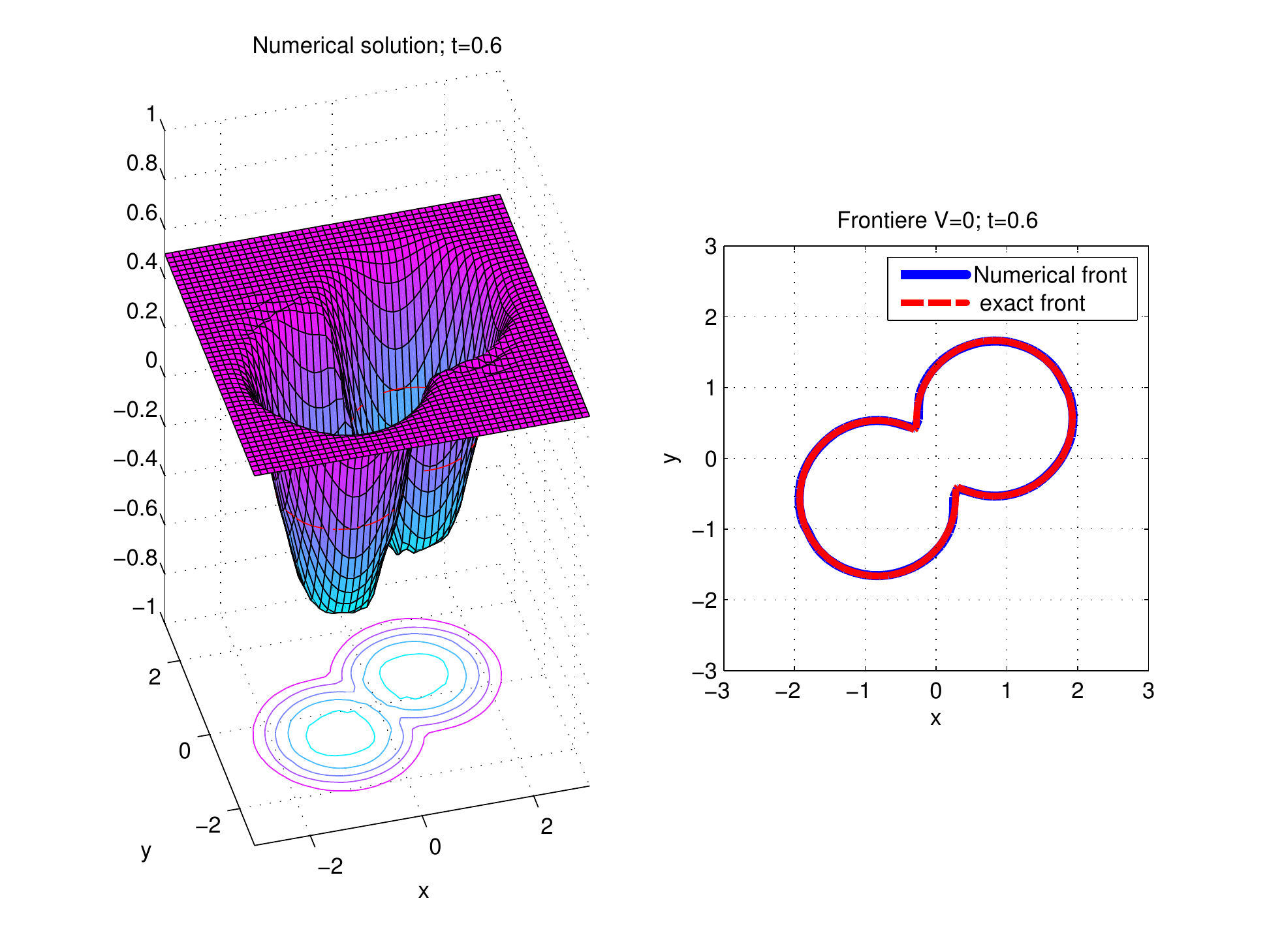}\\
Plots at time $t=0$ (top) and $t=0.75$ (bottom) for the filtered scheme and $M=50$ mesh points.
The figures to the right represent the zero-level sets.
\end{figure}
}\fi

\newcommand{\Examplesix}{Example 5}
\newcommand{\Exampleseven}{Example 6}
\newcommand{\Exampleheight}{Example 7}

\medskip\noindent

{\bf \Examplesix\  Steady eikonal equation.}
We consider a steady eikonal equation with Dirichlet boundary condition, which is taken from 
Abgrall~\cite{A09}:
\begin{subequations} \label{eq:A}
\be
  & & |v_x|=f(x)\quad \ x\in (0,1),\\
  & & v(0)=v(1)=0,
\ee
\end{subequations} 
where
$f(x)=3x^2+a$, with  $a=\frac{1-2x_0^3}{2x_0-1}$ and $x_0=\frac{\sqrt[3]{2}+2}{4\sqrt[3]{2}}$. 
Exact solution is known:
\be
v(x) := \left\{
  \begin{array}{l l}
    x^3+ax & \quad x\in[0,x_0], \\
    1+a-ax-x^3 & \quad  x\in[x_0,1].
  \end{array} \right.
\ee
The steady solution for \eqref{eq:A} can be considered as the limit $\disp\lim_{t\rightarrow \infty} v(t,x)$ 
where $v$ is the solution of the time marching corresponding form:
\begin{subequations} \label{eq:A1}
\be
  & & v_t+|v_x|=f(x)\quad \ x\in (0,1), \ t>0, \\
  & & v(t,0)=v(t,1)=0, \quad t>0.
\ee
\end{subequations} 
In this example the upwind monotone scheme is used:
\beno
  h^M(.)_j:=\frac{u_j^{n+1}-u^n_j}{\dt} - \max\{ \frac{u_j^n -u^n_{j-1}}{\dx} , \frac{u_j^n -u^n_{j+1}}{\dx}\} 
    - \dt f(x_j),
\eeno
the high--order scheme will be the centered scheme,
and the filtered scheme \eqref{eq:FS} will be used with $\eps=5\dx$.
The iterations are stopped when the difference between too successive time step is small enough or a fixed number of iterations is passed, i.e., in this example,
\be
  \|u^{n+1} - u^n\|_{L^\infty}:=\max_i |u^{n+1}_{i} -u^n_{i}| \leq 10^{-6}  \quad \mbox{or} \quad n\geq N_{max}:=5000.
\ee
As analyzed in~\cite{bok-fal-fer-gru-kal-zid} for $\eps$-monotone schemes, for a given mesh step, 
even if the iterations may not converge (because of the non monotony of the scheme), 
it can be shown to be close to a fixed
point after enough iterations.

\if{
\begin{table}[!hbtp]
\begin{center}
\begin{tabular}{|c|cc|cc|cc|}
\hline
   & \multicolumn{2}{|c|}{filtered ($\eps= 5\dx$)} & \multicolumn{2}{|c|}{centered} & \multicolumn{2}{|c|}{ENO2} \\
\hline
  $M$ &  error   & order  & error    & order  &  error   & order  \\
\hline
\hline 
\hline
\end{tabular}
\caption{(\Examplesix) Global errors of filtered scheme with RK2 in time.
\label{tab:ex6}
}
\end{center}
\end{table}
}\fi

\begin{table}[!hbtp]
\begin{center}
\begin{tabular}{|c|cc|cc||cc|}
\hline
   & \multicolumn{2}{|c|}{filtered} & \multicolumn{2}{|c|}{centered} & \multicolumn{2}{|c|}{filtered + ENO} \\
\hline
  $M$ &  error   & order  & error    & order  & error    & order  \\
\hline
\hline 
   50 & 2.16E-03 &   -   &    NaN &   - & 5.29E-03 &   -    \\ 
  100 & 7.14E-04 &  1.60 &    NaN &   - & 1.35E-03 &  1.97  \\ 
  200 & 2.17E-04 &  1.72 &    NaN &   - & 3.42E-04 &  1.98  \\ 
  400 & 6.32E-05 &  1.78 &    NaN &   - & 8.61E-05 &  1.99  \\ 
  800 & 2.17E-05 &  1.54 &    NaN &   - & 2.16E-05 &  2.00  \\ 
\hline
\end{tabular}
\caption{(\Examplesix) Global errors for filtered scheme, compared with the centered (unstable) scheme,
and a filtered ENO scheme.
\label{tab:ex6}
}
\end{center}
\end{table}

\if{
\begin{table}[!h]
\begin{center}
\begin{tabular}{|c|cc|cc|cc|}
\hline
   & \multicolumn{2}{|c|}{($L_1$-Error)} & \multicolumn{2}{|c|}{$L_2$-Error)} & \multicolumn{2}{|c|}{$L_\infty$-Error)} \\
\hline
  $M$ &  error   & order  & error    & order  &  error   & order  \\
\hline
\hline 
   50 & 1.90E-03 &   -    & 2.97E-03 &   -    & 5.29E-03 &   -    \\ 
  100 & 5.02E-04 &  1.92  & 7.81E-04 &  1.93  & 1.35E-03 &  1.97  \\ 
  200 & 1.28E-04 &  1.97  & 1.99E-04 &  1.97  & 3.42E-04 &  1.98  \\ 
  400 & 3.24E-05 &  1.98  & 5.04E-05 &  1.98  & 8.61E-05 &  1.99  \\ 
  800 & 8.15E-06 &  1.99  & 1.27E-05 &  1.99  & 2.16E-05 &  2.00  \\ 
   \hline
\end{tabular}
\caption{\Examplesix Global errors of ENO2 scheme  + filtered with RK2 in time.
\label{tab:ex4.1-3}
}
\end{center}
\end{table}
\begin{table}[!h]
\begin{center}
\begin{tabular}{|c|cc|cc|cc|}
\hline
   & \multicolumn{2}{|c|}{($L_1$-Error)} & \multicolumn{2}{|c|}{$L_2$-Error)} & \multicolumn{2}{|c|}{$L_\infty$-Error)} \\
\hline
  $M$ &  error   & order  & error    & order  &  error   & order  \\
\hline
\hline 
   50 & 3.46E-03 &   -    & 1.05E-02 &   -    & 7.04E-02 &   -    \\ 
  100 & 9.04E-04 &  1.94  & 3.67E-03 &  1.52  & 3.54E-02 &  0.99  \\ 
  200 & 2.30E-04 &  1.98  & 1.29E-03 &  1.51  & 1.78E-02 &  1.00  \\ 
  400 & 5.81E-05 &  1.98  & 4.53E-04 &  1.51  & 8.89E-03 &  1.00  \\ 
  800 & 1.46E-05 &  1.99  & 1.60E-04 &  1.50  & 4.45E-03 &  1.00  \\ 
\hline
\end{tabular}
\caption{(Example 6) Global errors of ENO alone scheme with RK2 in time. (only ENO2)
\label{tab:ex6-eno}
}
\end{center}
\end{table}
}\fi


\begin{figure}[!h]
\vspace{-3cm}
\begin{center}
\includegraphics[width=0.5\textwidth]{\figs/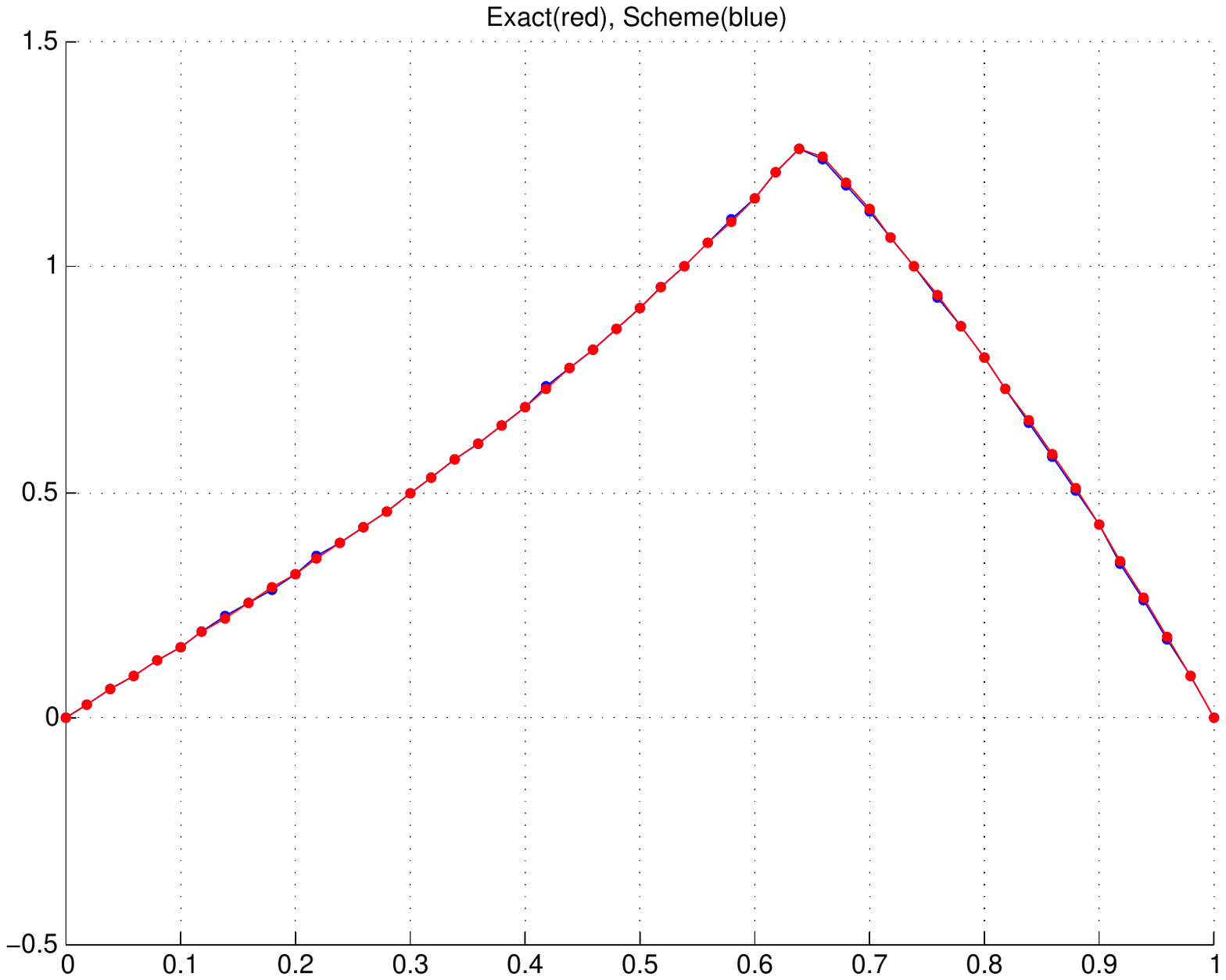}
\caption{(\Examplesix) Filtered scheme for a steady equation, with $M=50$ mesh points.}
\end{center}
\vspace{-3cm}
\end{figure}

\medskip\noindent
{\bf \Exampleseven\ Advection with an obstacle.}
Here we consider an obstacle problem,
which is taken from~\cite{Boka-cheng-shu-13}:
\be
  && min(v_t+v_x,\ v-g(x))=0,\quad  t >0, x\in[-1,1],\\
  && v_0(x)= 0.5+sin(\pi x) \quad x\in[-1, 1],
\ee
together with periodic boundary condition. The obstacle function is $g(x):= sin(\pi x)$.
In this case exact solution is given by:
\be
v(t,x) := \left\{
  \begin{array}{l l}
    \max(v_0(x-at), g(x))& ~ \text{if}~ t< \frac{1}{3} \\
    \max(v_0(x-at), g(x), -1_{x\in\ [0.5,1]} )& ~ \text{if} ~ t \in [ \frac{1}{3},\frac{5}{6}], \\
    \max(v_0(x-at), g(x), 1_{x\in\ [-1,t-\frac{5}{6}] \cup [0.5,1]} )& ~ \text{if} ~ t \in [\frac{5}{6}, 1], \\
  \end{array} \right.
\ee
Results are given in Table~\ref{tab:ex7}, for terminal time $t=0.5$.
Errors are computed away from singular points, i.e., 
in the region 
$[-1,1]  \setminus \big( \cup _{i=1,3}[s_i-\delta, s_i+\delta]\big)$
(where $s_1=-0.1349733, s_2=0.5$ and $s_3=2/3$ are the three singular points.
Filtered scheme is numerically of second order (ENO gives comparable results here).

\begin{table}[\tablepos]
\begin{center}
\begin{tabular}{|cc|cc|cc|cc|}
\hline
  \multicolumn{2}{|c|}{Errors} 
   & \multicolumn{2}{|c|}{filtered $\eps= 5\dx$} & \multicolumn{2}{|c|}{centered} & \multicolumn{2}{|c|}{ENO2} \\
\hline
  $M$ & $N$  &  error   & order  & error    & order  &  error   & order  \\
\hline
\hline 
   40 &   20  & 7.93E-03 &  2.03  & 1.63E-02 &  1.54  & 2.14E-02 &  1.59  \\ 
   80 &   40  & 1.84E-03 &  2.10  & 2.98E-02 & -0.87  & 7.75E-03 &  1.46  \\ 
  160 &   80  & 3.92E-04 &  2.24  & 1.46E-02 &  1.03  & 1.07E-03 &  2.86  \\ 
  320 &  160  & 9.67E-05 &  2.02  & 8.02E-03 &  0.86  & 2.72E-04 &  1.97  \\ 
  640 &  320  & 2.40E-05 &  2.01  & 4.10E-03 &  0.97  & 6.92E-05 &  1.98  \\ 
\hline
\end{tabular}
\caption{(\Exampleseven) $L^\infty$ errors away from singular points, 
for filtered scheme, centered scheme, and second order ENO scheme.
\label{tab:ex7}
}
\end{center}
\end{table}

\if{
\begin{table}[\tablepos]
\begin{center}
\begin{tabular}{|cc|cc|cc|cc|}
\hline
  \multicolumn{2}{|c|}{Errors} 
   & \multicolumn{2}{|c|}{($L_1$-Error)} & \multicolumn{2}{|c|}{$L_2$-Error)} & \multicolumn{2}{|c|}{$L_\infty$-Error)} \\
\hline
  $M$ & $N$  &  error   & order  & error    & order  &  error   & order  \\
\hline
\hline 
   40  &   25  & 3.56E-03 &     -      & 3.96E-03 &   -        & 5.43E-03 &   -    \\ 
   80  &   49  & 8.70E-04 &  2.03  & 9.77E-04 &  2.02  & 1.53E-03 &  1.83  \\ 
  160 &   98  & 2.15E-04 &  2.02  & 2.40E-04 &  2.03  & 3.47E-04 &  2.14  \\ 
  320 &  196 & 5.35E-05 &  2.01  & 5.96E-05 &  2.01  & 9.53E-05 &  1.86  \\ 
  640 &  391 & 1.33E-05 &  2.01  & 1.49E-05 &  2.00  & 2.56E-05 &  1.90  \\ 
 \hline
\end{tabular}
\caption{ ({\Exampleseven}.2)
Filtered scheme  ($5\dx$) with RK2 in time.
\label{tab:ex72}
}
\end{center}
\end{table}
}\fi

\begin{figure}[!h]
\vspace{-1.0cm}
\includegraphics[width=0.3\textwidth]{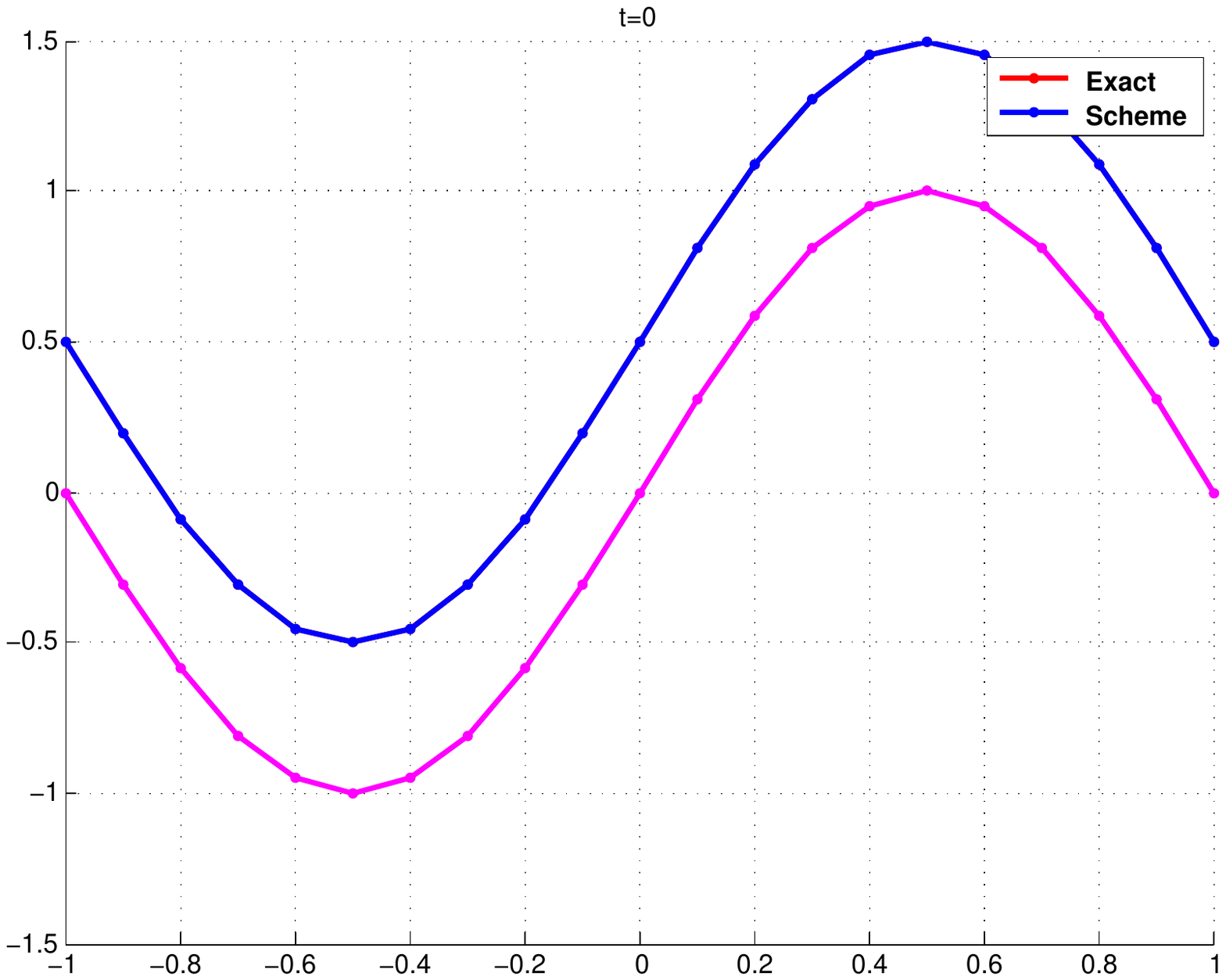}
\includegraphics[width=0.3\textwidth]{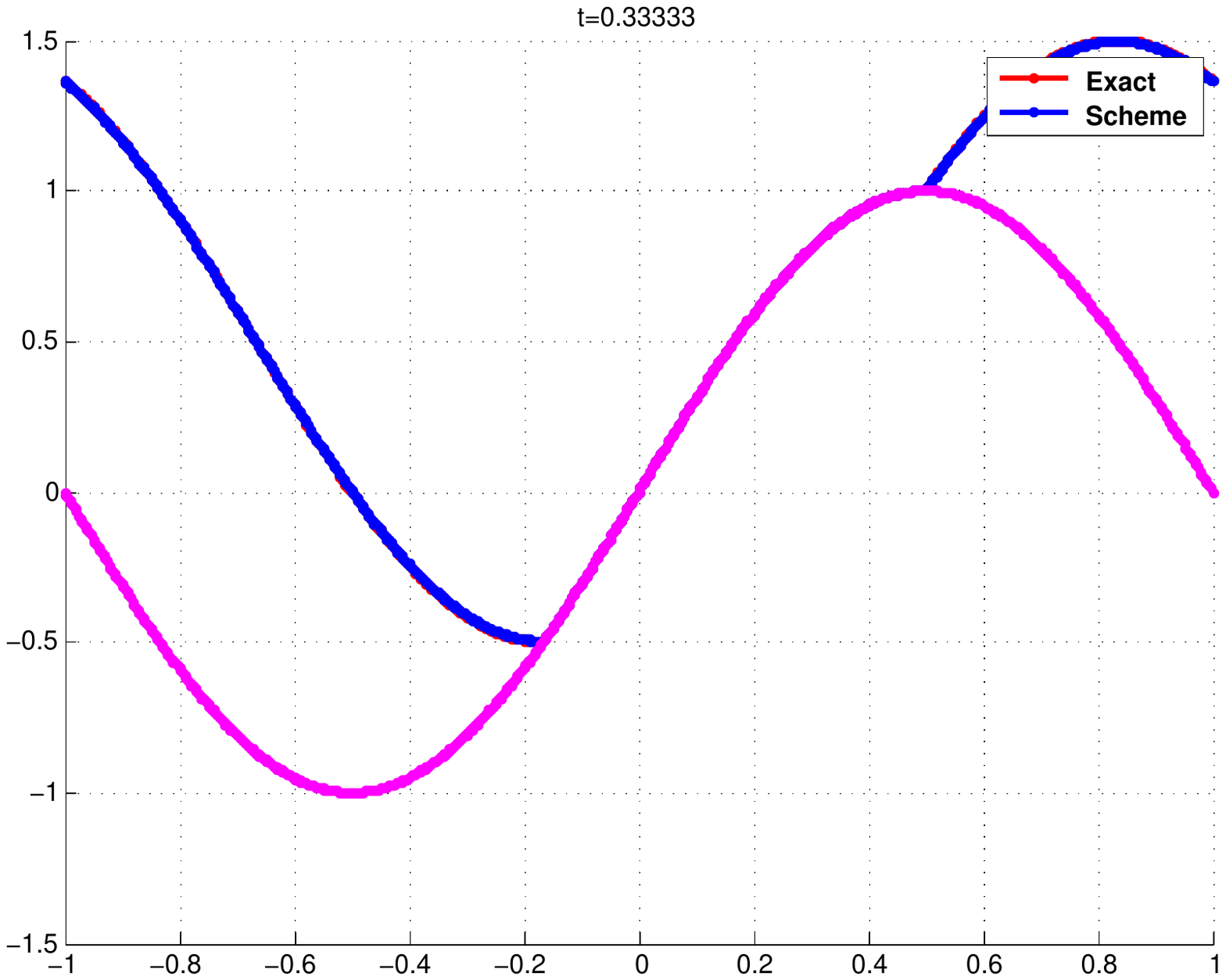}
\includegraphics[width=0.3\textwidth]{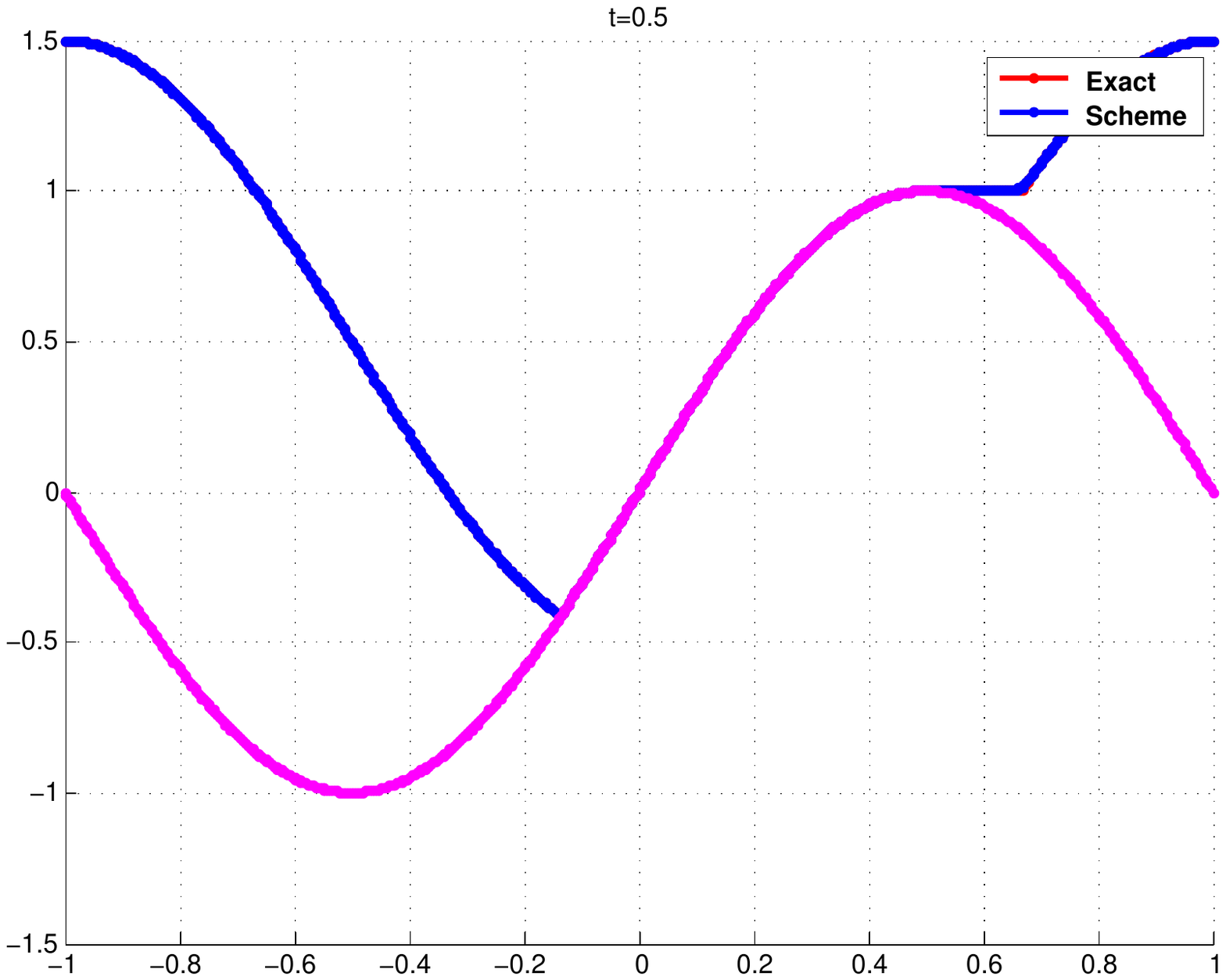}
\vspace{-1.0cm}
\caption{(\Exampleseven) Plots at T=0(initial data), T=0.3, T=0.5.}
\end{figure}

\medskip\noindent
{\bf \Exampleheight\ Eikonal with an obstacle.}
We consider an Eikonal equation with an obstacle term, also taken from~\cite{Boka-cheng-shu-13}:
\be
&& min(v_t+|v_x|, v-g(x))=0,\quad  t >0, x\in[-1,1],\\
&& v_0(x)= 0.5+sin(\pi x)\quad x\in[-1, 1],
\ee
with periodic boundary condition on $(-1,1)$ and $g(x)= sin(\pi x)$.
In this case the exact solution is given by:
\be
   v(t,x)=\max(\bar v(t,x), g(x)).
\ee
where $\bar v$ is the solution of the Eikonal equation $v_t+ |v_x|=0$.
The formula $\bar v(t,x)= \min_{y\in[x-t,x+t]}v_0(y)$ holds, which simplifies to
\be
  v(t,x) := \left\{
  \begin{array}{l l}
    v_0(x+t)& ~\text{if} ~ x< -0.5-t \\
     -0.5   & ~\text{if} ~ x \in [ -0.5-t, -0.5+t], \\
    \min(v_0(x-t), v_0(x+t))& ~ \text{if} ~ x \geq -0.5+t, \\
  \end{array} \right.
\ee

Results are given in Table~\ref{tab:ex8} for terminal time $T=0.2$.
Plots are also shown in Figure~\ref{fig:ex8} for different times
(for $t\geq \frac{1}{3}$ solution remains unchanged).


\if{
\begin{table}[\tablepos]
\begin{center}
\begin{tabular}{|cc|cc|cc|cc|}
\hline
  Errors
   & \multicolumn{2}{|c|}{filtered} & \multicolumn{2}{|c|}{ENO2} \\
\hline
  $M$ &   error   & order  & error    & order \\
\hline
\hline 
   40 &   2.35E-03 &   -    & 3.74E-03 &   -    & 9.59E-03 &   -    \\ 
   80 &   3.88E-04 &  2.60  & 6.26E-04 &  2.58  & 2.79E-03 &  1.78  \\ 
  160 &   7.57E-05 &  2.36  & 1.13E-04 &  2.47  & 4.12E-04 &  2.76  \\ 
  320 &   1.59E-05 &  2.25  & 2.26E-05 &  2.32  & 1.04E-04 &  1.99  \\ 
  640 &   3.73E-06 &  2.09  & 5.50E-06 &  2.04  & 3.22E-05 &  1.69  \\ 
  \hline
\end{tabular}
\caption{\label{tab:ex8}
Filtered scheme at time $t=0.2$}
\end{center}
\end{table}

\begin{table}[\tablepos]
\begin{center}
\begin{tabular}{|cc|cc|cc|cc|}
\hline
  \multicolumn{2}{|c|}{Errors} 
   & \multicolumn{2}{|c|}{($L_1$-Error)} & \multicolumn{2}{|c|}{$L_2$-Error)} & \multicolumn{2}{|c|}{$L_\infty$-Error)} \\
\hline
  $M$ & $N$  &  error   & order  & error    & order  &  error   & order  \\
\hline
\hline 
   40 &    8  & 5.87E-03 &   -    & 6.85E-03 &   -    & 1.31E-02 &   -    \\ 
   80 &   16  & 1.52E-03 &  1.95  & 2.12E-03 &  1.69  & 5.53E-03 &  1.24  \\ 
  160 &   32  & 3.98E-04 &  1.93  & 6.80E-04 &  1.64  & 2.25E-03 &  1.30  \\ 
  320 &   64  & 1.04E-04 &  1.94  & 2.18E-04 &  1.64  & 9.20E-04 &  1.29  \\ 
  640 &  128  & 2.67E-05 &  1.96  & 6.96E-05 &  1.65  & 3.77E-04 &  1.29  \\ 
     \hline
\end{tabular}
\caption{ ENO scheme (second order ) with RK2 in time and terminal time T=0.2.}
\end{center}
\end{table}
}\fi

\begin{table}[\tablepos]
\begin{center}
\begin{tabular}{|c|cc|cc|}
\hline
  Errors
   & \multicolumn{2}{|c|}{filtered} & \multicolumn{2}{|c|}{ENO2} \\
\hline
  $M$ &   error   & order  & error    & order \\
\hline
\hline 
   40 &   3.74E-03 &   -    & 6.85E-03 &   -    \\
   80 &   6.26E-04 &  2.58  & 2.12E-03 &  1.69  \\
  160 &   1.13E-04 &  2.47  & 6.80E-04 &  1.64  \\
  320 &   2.26E-05 &  2.32  & 2.18E-04 &  1.64  \\
  640 &   5.50E-06 &  2.04  & 6.96E-05 &  1.65  \\
\hline
\end{tabular}
\caption{\label{tab:ex8}
Filtered scheme and ENO scheme at time $t=0.2$}
\end{center}
\end{table}

\begin{figure}[!h]
\vspace{-1cm}
\includegraphics[width=0.3\textwidth]{\figsnew/obstacle}
\includegraphics[width=0.3\textwidth]{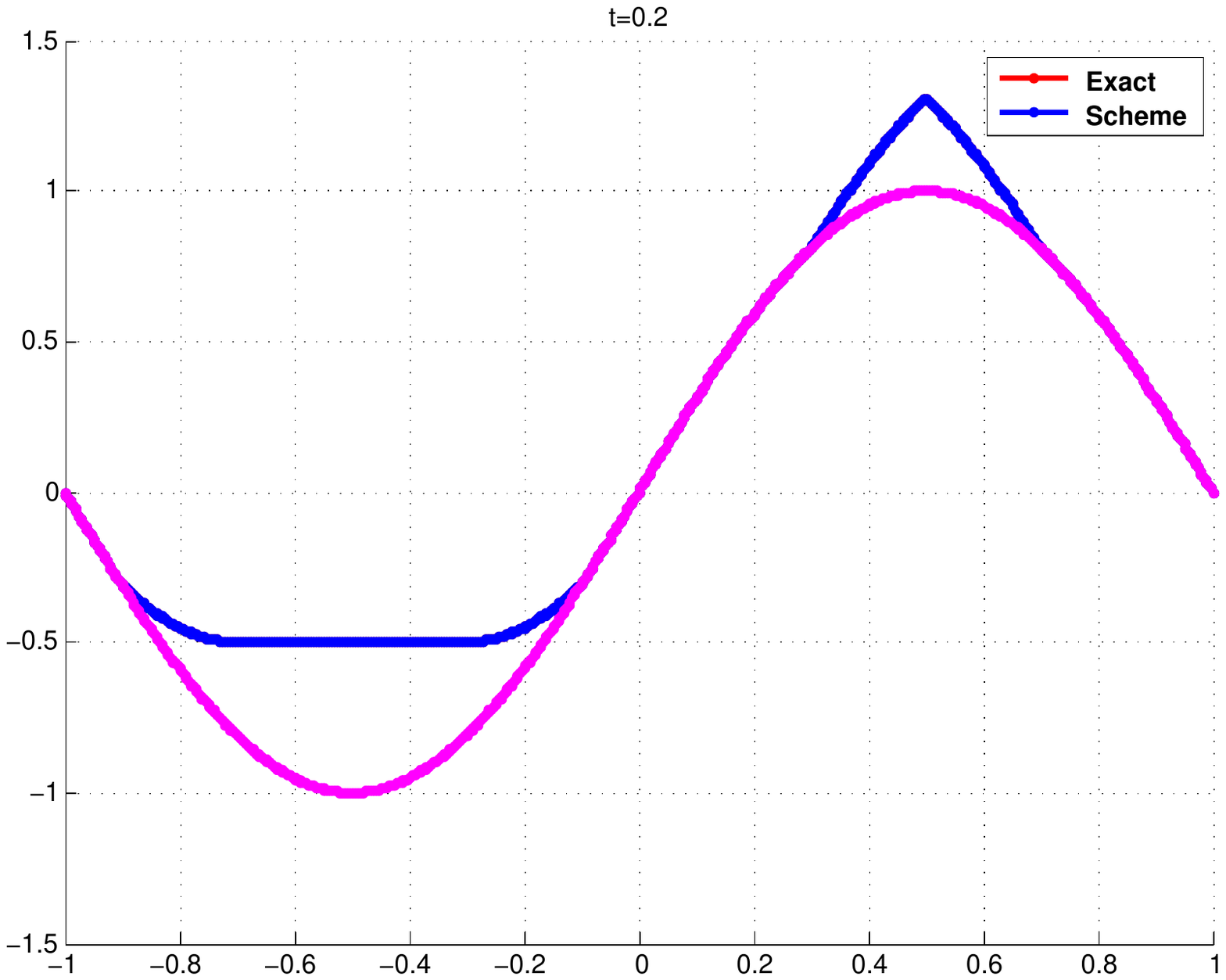}
\includegraphics[width=0.3\textwidth]{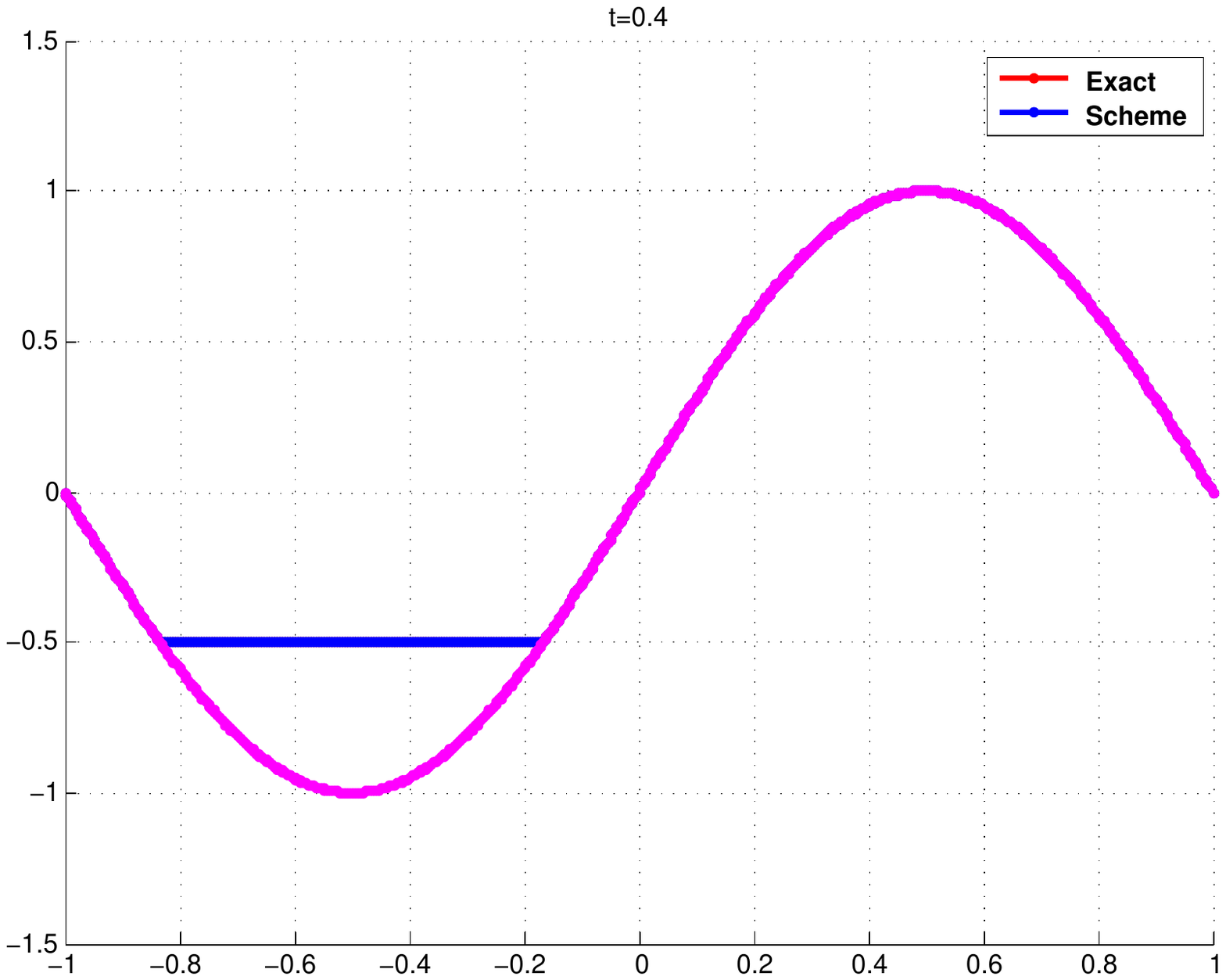}
\vspace{-1cm}
\caption{\label{fig:ex8} (\Exampleheight) Plots at times $t=0$, $t=0.2$ and $t=0.4$.
The dark line is the numerical solution, similar to the exact solution, and the ligth line is the obstacle function.}
\end{figure}

%

\section{Conclusion}
We propose a ``filtered" scheme which behaves as a high--order scheme when the solution is smooth and as a low order
monotone scheme otherwise.
It has a simple presentation that is easy to implement. 
Rigorous error bounds hold, of the same order as the Crandall-Lions estimates in $\sqrt{\dx}$ where $\dx$ is the mesh size.
In the case the solution is smooth a high-order consistency error estimate also holds.
We show on several numerical examples the ability of the scheme to stabilize an otherwise
unstable scheme, and also we observe a precision similar to a second order ENO scheme
on basic linear and non linear examples. 

On going works concern the application of the present approach to some front propagation equations.

\appendix

\section{An essentially non-oscillatory (ENO) scheme of second order} \label{app:A} 

We recall here a simple ENO method of order two
based on the work of Osher and Shu~\cite{OS91} for Hamilton Jacobi equation
(the ENO method was designed by Harten et al.~\cite{HEOC87}
for the approximation solution of non-linear conservation law).

Let $m$ be the minmod function defined by
\be\label{linEquGrad}
m(a,b)=\left\{
  \begin{array}{ll}
    \displaystyle   a  \quad {\rm if}\, |a|\leq |b|,~ ab >0\\
    \displaystyle   b  \quad {\rm if}\, |b| <|b|,~ ab>0\\
    \displaystyle   0  \quad {\rm if}\, ab \leq 0
  \end{array}
\right.
\ee 
(other functions can be considered such as $m(a,b)=a$ if $|a|\leq |b|$ and $m(a,b)=b$ otherwise).
Let $D^\pm u_j=\pm (u_{j\pm 1}-u_j)/\dx$ and 
$$
  D^2 u_j : = \frac{u_{j+1}-2 u_j + u_{j-1}}{\dx^2}.
$$
Then the right and left ENO approximation of the derivative can be defined by
\beno
  \bar D^{\pm}u_{j}= D^{\pm}u_j \mp \frac{1}{2} \dx\  m(D^2 u_{j},D^2 u_{j\pm 1} )
\eeno
and the ENO (Euler forward) scheme by 
$$
  S_0(u)_j:= u_j - \tau h^M(x_j,\,\bar D^-u_j,\, \bar D^+u_j).
$$
The corresponding RK2 scheme can then be defined by $S(u)=\frac{1}{2}(u + S_0(S_0(u)))$.

\if{
{\bf{TVD RK3 scheme }} Here we are recalling  third order TVD RK  scheme 
\beno
  & & u^{n,1}_i:= u^n_i - \dt h(x_i,D^\mp u^n_i).\\
  & &u^{n,2}_{i}:= \frac{3}{4}(u^n_i +\frac{1}{4}u^{n,1}-\dt h(x_i,D^\mp u^{n,1}_i).\\
  && u^{n+1}_i= \frac{1}{3}u_i^n+\frac{2}{3}u^{n,2}-\frac{2}{3}\dt h(x_i,D^\mp u^{n,2}_i),\\
\eeno
}\fi


\bibliography{Bibliography}{}
\bibliographystyle{plain}

\end{document}